%% file: MVPolarities.tex
\documentclass{article}
\usepackage{amsthm}
\newtheorem{theorem}{Theorem}[section]

\newtheorem{lemma}[theorem]{Lemma}

\newtheorem{definition}[theorem]{Definition}

\newcommand{\Prop}{\mathsf{Prop}}

\newcommand{\val}[1]{[\![{#1}]\!]}
\newcommand{\descr}[1]{(\![{#1}]\!)}

\renewcommand{\phi}{\varphi}

\input{packages}
\usepackage{enumitem}

\title{\bf The logic of vague categories}

\author{Willem Conradie$^{a}$, Alessandra Palmigiano$^{b,c}$, Claudette Robinson$^{c}$,\\ 
	Apostolos Tzimoulis$^{b}$, Nachoem M.~Wijnberg$^{d,e}$\\ 
$^a$School of Mathematics, University of the Witwatersrand \\
$^{b}$School of Business and Economics, Vrije Universiteit, Amsterdam\\
$^c$Department of Pure and Applied Mathematics, University of Johannesburg\\
$^d$ Amsterdam Business School, University of Amsterdam\\
$^e$ College of Business and Economics,  University of Johannesburg
} 

\begin{document}
\maketitle
\begin{abstract}
We introduce a complete many-valued semantics for basic normal lattice-based modal logic. This  relational semantics is grounded on many-valued formal contexts from Formal Concept Analysis. We discuss an interpretation and possible applications of this logical framework for categorization theory to the formal analysis of multi-market competition.
\end{abstract}

\section{Introduction} 
This paper pertains to a line of research on the semantics of normal LE-logics \cite{CoPa:non-dist, conradie2016constructive}, which are the logics canonically associated with varieties of normal lattice expansions (LEs) \cite{gehrke2001bounded}, (i.e.~general, not necessarily distributive, lattices endowed with extra operations of any arity each of which is coordinatewise either finitely join-preserving/meet-reversing, or finitely meet-preserving/join-reversing).  Thanks to duality theory and  canonical extensions, basic normal LE-logics of arbitrary signatures and  a large class of their axiomatic extensions can be uniformly endowed with complete relational semantics of different kinds, of which  those of interest to the present paper are relational structures based on {\em formal contexts} (aka {\em polarities}) \cite{gehrke2006generalized, galatos2013residuated, conradie2016categories,Tarkpaper, greco2018algebraic}. 
Polarity-based structures have yielded important theoretical contributions in the algebraic proof theory \cite{greco2018algebraic} and in the model theory \cite{conradie2018goldblatt} of LE-logics, and have also contributed to illuminate the conceptual significance  
of LE-logics, thereby breaking the ground for novel applications. Specifically, in \cite{conradie2016categories}, the basic non-distributive modal logic and some of its axiomatic extensions are interpreted as  {\em epistemic logics of categories and concepts}, and in \cite{Tarkpaper}, the corresponding `common knowledge'-type construction is used to give an epistemic-logical formalization of the notion of {\em prototype} of a category; in \cite{roughconcepts, ICLA2019paper}, polarity-based semantics for non-distributive modal logic is proposed as an encompassing framework for the integration of rough set theory \cite{pawlak} and formal concept analysis \cite{ganter2012formal}, and in this context, the basic non-distributive modal logic is interpreted as  the logic of {\em rough concepts}. Other different but related semantics of the same logic have been introduced and explored in \cite{conradie2015relational, graph-based-wollic}, also regarding the many-valued semantic setting  \cite{eusflat, socio-political}.   

In this paper, we pursue  the mathematical and conceptual investigation of  the many-valued polarity-based semantics for the basic modal non-distributive logic adumbrated in \cite[Section 7.2]{roughconcepts}. Specifically, the main technical contribution of the present paper is the proof that the basic non-distributive modal logic is complete w.r.t.~the class of many-valued enriched formal contexts.  The main conceptual contribution is the discussion of
the potential of this logical framework to provide a formal ground on which to build the theory of multi-market competition in management science.   
\section{Preliminaries}
This section adapts material from \cite[Section 2.1]{graph-based-wollic},  \cite[Section 7.2]{roughconcepts}, \cite[Sections 3 and 4]{eusflat} and \cite[Sections 3 and 4]{socio-political}.
\subsection{Basic normal nondistributive modal logic}
\label{sec:logics}
Let $\Prop$ be a (countable or finite) set of atomic propositions. The language $\mathcal{L}$ of the {\em basic normal nondistributive modal logic} is defined as follows:
\[ \varphi := \bot \mid \top \mid p \mid  \varphi \wedge \varphi \mid \varphi \vee \varphi \mid \Box \varphi \mid  \Diamond\varphi,\] 
where $p\in \Prop$. 
The {\em basic}, or {\em minimal normal} $\mathcal{L}$-{\em logic} is a set $\mathbf{L}$ of sequents $\phi\vdash\psi$  with $\phi,\psi\in\mathcal{L}$, containing the following axioms:
		{\small{
			\begin{align*}
				&p\vdash p, && \bot\vdash p, && p\vdash \top, & &  &\\
				&p\vdash p\vee q, && q\vdash p\vee q, && p\wedge q\vdash p, && p\wedge q\vdash q, &\\
				& \top\vdash \Box \top, && \Box p\wedge \Box q \vdash \Box ( p\wedge q),
           && \Diamond \bot\vdash \bot, &&
                \Diamond ( p \vee q) \vdash \Diamond p \vee \Diamond  q &
			\end{align*}
			}}
	%
	%
		and closed under the following inference rules:
		{\small{
		\begin{displaymath}
			\frac{\phi\vdash \chi\quad \chi\vdash \psi}{\phi\vdash \psi}
			\quad
			\frac{\phi\vdash \psi}{\phi\left(\chi/p\right)\vdash\psi\left(\chi/p\right)}
			\quad
			\frac{\chi\vdash\phi\quad \chi\vdash\psi}{\chi\vdash \phi\wedge\psi}
			\quad
			\frac{\phi\vdash\chi\quad \psi\vdash\chi}{\phi\vee\psi\vdash\chi}
			\end{displaymath}
			\begin{displaymath}
			\frac{\phi\vdash\psi}{\Box \phi\vdash \Box \psi}
\quad
\frac{\phi\vdash\psi}{\Diamond \phi\vdash \Diamond \psi}
		\end{displaymath}
		}}
An {\em $\mathcal{L}$-logic} is any  extension of $\mathbf{L}$  with $\mathcal{L}$-axioms $\phi\vdash\psi$.
	
\subsection{Many-valued enriched formal contexts}
Throughout this paper,  we let $\mathbf{A} = (D, 1, 0, \vee, \wedge, \otimes, \to)$ denote an arbitrary but fixed 
complete frame-distributive and dually frame-distributive, 
commutative and associative residuated lattice  (understood as the algebra of truth-values) such that $1\to \alpha = \alpha$ for every $\alpha\in D$.
For every set $W$, an $\mathbf{A}$-{\em valued subset}  (or $\mathbf{A}$-{\em subset}) of $W$ is a map $u: W\to \mathbf{A}$.  We let $\mathbf{A}^W$ denote the set of all $\mathbf{A}$-subsets. Clearly, $\mathbf{A}^W$ inherits the algebraic structure of $\mathbf{A}$ by defining the operations and the order pointwise. The $\mathbf{A}$-{\em subsethood} relation between elements of $\mathbf{A}^W$ is the map $S_W:\mathbf{A}^W\times \mathbf{A}^W\to \mathbf{A}$ defined as $S_W(f, g) :=\bigwedge_{z\in W }(f(z)\rightarrow g(z)) $. For every $\alpha\in \mathbf{A}$, 
let $\{\alpha/ w\}: W\to \mathbf{A}$ be defined by $v\mapsto \alpha$ if $v = w$ and $v\mapsto \bot^{\mathbf{A}}$ if $v\neq w$. Then, for every $f\in \mathbf{A}^W$,
\begin{equation}\label{eq:MV:join:generators}
f = \bigvee_{w\in W}\{f(w)/ w\}.
\end{equation}
 When $u, v: W\to \mathbf{A}$ and $u\leq v$ w.r.t.~the pointwise order, we write $u\subseteq v$.
An $\mathbf{A}$-{\em valued relation} (or $\mathbf{A}$-{\em relation}) is a map $R: U \times W \rightarrow \mathbf{A}$. Two-valued relations can be regarded as  $\mathbf{A}$-relations. In particular for any set $Z$, we let $\Delta_Z: Z\times Z\to \mathbf{A}$ be defined by $\Delta_Z(z, z') = \top$ if $z = z'$ and $\Delta_Z(z, z') = \bot$ if $z\neq z'$. 
Any $\mathbf{A}$-valued relation $R: U \times W \rightarrow \mathbf{A}$ induces  maps $R^{(0)}[-] : \mathbf{A}^W \rightarrow \mathbf{A}^U$ and $R^{(1)}[-] : \mathbf{A}^U \rightarrow \mathbf{A}^W$ defined as follows: for every $f: U \to \mathbf{A}$ and every $u: W \to \mathbf{A}$,
\begin{center}
	\begin{tabular}{r l }
$R^{(1)}[f]:$ & $ W\to \mathbf{A}$\\
		
		& $ x\mapsto \bigwedge_{a\in U}(f(a)\rightarrow R(a, x))$\\
	&\\	
 $R^{(0)}[u]: $ & $U\to \mathbf{A} $\\
		& $a\mapsto \bigwedge_{x\in W}(u(x)\rightarrow R(a, x))$\\
	\end{tabular}
\end{center}

A {\em formal}  $\mathbf{A}$-{\em context}\footnote{ In the crisp setting, a {\em formal context} \cite{ganter2012formal}, or {\em polarity},  is a structure $\mathbb{P} = (A, X, I)$ such that $A$ and $X$ are sets, and $I\subseteq A\times X$ is a binary relation. Every such $\mathbb{P}$ induces maps $(\cdot)^\uparrow: \mathcal{P}(A)\to \mathcal{P}(X)$ and $(\cdot)^\downarrow: \mathcal{P}(X)\to \mathcal{P}(A)$, respectively defined by the assignments $B^\uparrow: = I^{(1)}[B]$ and $Y^\downarrow: = I^{(0)}[Y]$. A {\em formal concept} of $\mathbb{P}$ is a pair 
$c = (\val{c}, \descr{c})$ such that $\val{c}\subseteq A$, $\descr{c}\subseteq X$, and $\val{c}^{\uparrow} = \descr{c}$ and $\descr{c}^{\downarrow} = \val{c}$.   The set $L(\mathbb{P})$  of the formal concepts of $\mathbb{P}$ can be partially ordered as follows: for any $c, d\in L(\mathbb{P})$, \[c\leq d\quad \mbox{ iff }\quad \val{c}\subseteq \val{d} \quad \mbox{ iff }\quad \descr{d}\subseteq \descr{c}.\]
With this order, $L(\mathbb{P})$ is a complete lattice, the {\em concept lattice} $\mathbb{P}^+$ of $\mathbb{P}$. Any complete lattice $\mathbb{L}$ is isomorphic to the concept lattice $\mathbb{P}^+$ of some polarity $\mathbb{P}$.} or $\mathbf{A}$-{\em polarity} (cf.~\cite{belohlavek}) is a structure $\mathbb{P} = (A, X, I)$ such that $A$ and $X$ are sets and $I: A\times X\to \mathbf{A}$. Any formal $\mathbf{A}$-context induces  maps $(\cdot)^{\uparrow}: \mathbf{A}^A\to \mathbf{A}^X$ and $(\cdot)^{\downarrow}: \mathbf{A}^X\to \mathbf{A}^A$ given by $(\cdot)^{\uparrow} = I^{(1)}[\cdot]$ and $(\cdot)^{\downarrow} = I^{(0)}[\cdot]$. 
%
These maps are such that, for every $f\in \mathbf{A}^A$ and every $u\in \mathbf{A}^X$,
\[S_A(f, u^{\downarrow}) = S_X(u, f^{\uparrow}),\]
that is, the pair of maps $(\cdot)^{\uparrow}$ and $(\cdot)^{\downarrow}$ form an $\mathbf{A}$-{\em Galois connection}.
In \cite[Lemma 5]{belohlavek}, it is shown that every  $\mathbf{A}$-Galois connection arises from some formal  $\mathbf{A}$-context. A {\em formal}  $\mathbf{A}$-{\em concept} of $\mathbb{P}$ is a pair $(f, u)\in \mathbf{A}^A\times \mathbf{A}^X$ such that $f^{\uparrow} = u$ and $u^{\downarrow} = f$. It follows immediately from this definition that if $(f, u)$ is a formal $\mathbf{A}$-concept, then $f^{\uparrow \downarrow} = f$ and $u^{\downarrow\uparrow} = u$, that is, $f$ and $u$ are {\em stable}. The set of formal $\mathbf{A}$-concepts can be partially ordered as follows:
\[(f, u)\leq (g, v)\quad \mbox{ iff }\quad f\subseteq g \quad \mbox{ iff }\quad v\subseteq u. \]
Ordered in this way, the set of the formal  $\mathbf{A}$-concepts of $\mathbb{P}$ is a complete lattice, which we denote $\mathbb{P}^+$. 

	An {\em enriched formal $\mathbf{A}$-context} (cf.~\cite[Section 7.2]{roughconcepts}) is a structure  $\mathbb{F} = (\mathbb{P}, R_\Box, R_\Diamond)$ such that $\mathbb{P} = (A, X, I)$ is a formal  $\mathbf{A}$-context and $R_\Box: A\times X\to \mathbf{A}$ and $R_\Diamond: X\times A\to \mathbf{A}$ are $I$-{\em compatible}, i.e.~$R_{\Box}^{(0)}[\{\alpha / x\}]$, $R_{\Box}^{(1)}[\{\alpha / a\}]$,  $R_{\Diamond}^{(0)}[\{\alpha / a\}]$ and $R_{\Diamond}^{(1)}[\{\alpha / x\}]$ are stable for every $\alpha \in \mathbf{A}$, $a \in A$ and $x \in X$.   
	The {\em complex algebra} of an  enriched formal $\mathbf{A}$-context $\mathbb{F} = (\mathbb{P}, R_\Box, R_\Diamond)$ is the algebra $\mathbb{F}^{+} = (\mathbb{P}^{+}, [R_{\Box}], \langle R_{\Diamond} \rangle )$ where $[R_{\Box}], \langle R_{\Diamond} \rangle : \mathbb{P}^{+} \to \mathbb{P}^{+}$ are defined by the following assignments: for every $c = (\val{c}, \descr{c}) \in \mathbb{P}^{+}$, 
	\begin{center}
	\begin{tabular}{l cl}
		$[R_{\Box}]c$ & $ =$&$  (R_{\Box}^{(0)}[\descr{c}], (R_{\Box}^{(0)}[\descr{c}])^{\uparrow})$\\
		$ \langle R_{\Diamond} \rangle c $ & $ =$&$  ((R_{\Diamond}^{(0)}[\val{c}])^{\downarrow}, R_{\Diamond}^{(0)}[\val{c}])$.\\
	\end{tabular}
	\end{center}

\begin{lemma}\label{prop:fplus}
(cf.~\cite[Lemma 15]{roughconcepts}) If $\mathbb{F} = (\mathbb{X}, R_{\Box}, R_{\Diamond})$	is an enriched formal   $\mathbf{A}$-context, $\mathbb{F}^+ = (\mathbb{X}^+, [R_{\Box}], \langle R_{\Diamond}\rangle)$ is a complete  normal lattice expansion such that $[R_\Box]$ is completely meet-preserving and $\langle R_\Diamond\rangle$ is completely join-preserving.
\end{lemma}

\subsection{Many-valued polarity-based models}
Let $\mathcal{L}$ be the language of Section \ref{sec:logics}. 
\begin{definition}
	A {\em conceptual}  $\mathbf{A}$-{\em model} over a  set $\mathsf{AtProp}$ of atomic propositions is a tuple $\mathbb{M} = (\mathbb{F}, V)$ such that $\mathbb{F} = (A, X, I, R_\Box, R_\Diamond)$ is an enriched formal $\mathbf{A}$-context and $V: \mathsf{AtProp}\to \mathbb{F}^+$. For every $p\in \mathsf{AtProp}$, let $V(p): = (\val{p}, \descr{p})$, where $\val{p}: A\to \mathbf{A}$ and $\descr{p}: X\to\mathbf{A}$, and $\val{p}^\uparrow = \descr{p}$ and $\descr{p}^\downarrow = \val{p}$.
	Letting $\mathcal{L}$ denote the $\{\Box, \Diamond\}$ modal language  over $\mathsf{AtProp}$, 
	every $V$ as above has a unique homomorphic extension, also denoted $V: \mathcal{L} \to \mathbb{F}^+$, defined as follows:
	\begin{center}
		\begin{tabular}{r c l}
			$V(p)$ & = & $(\val{p}, \descr{p})$\\
			$V(\top)$ & = & $(\top^{\mathbf{A}^A}, (\top^{\mathbf{A}^A})^\uparrow)$\\
			$V(\bot)$ & = & $((\top^{\mathbf{A}^X})^\downarrow, \top^{\mathbf{A}^X})$\\
			$V(\phi\wedge \psi)$ & = & $(\val{\phi}\wedge\val{\psi}, (\val{\phi}\wedge\val{\psi})^\uparrow)$\\
			$V(\phi\vee \psi)$ & = & $((\descr{\phi}\wedge\descr{\psi})^\downarrow, \descr{\phi}\wedge\descr{\psi})$\\
			$V(\Box\phi)$ & = & $(R^{(0)}_\Box[\descr{\phi}], (R^{(0)}_\Box[\descr{\phi}])^\uparrow)$\\
			$V(\Diamond\phi)$ & = & $((R^{(0)}_\Diamond[\val{\phi}])^\downarrow, R^{(0)}_\Diamond[\val{\phi}])$\\
		\end{tabular}
	\end{center}
	which in its turn induces  $\alpha$-{\em membership relations} for each $\alpha\in \mathbf{A}$ (in symbols: $\mathbb{M}, a\Vdash^\alpha \phi$), and $\alpha$-{\em description relations} for each $\alpha\in \mathbf{A}$ (in symbols: $\mathbb{M}, x\succ^\alpha \phi$)---cf.~discussion in Section \ref{sec:logics}---such that for every $\phi\in \mathcal{L}$,
	\[\mathbb{M}, a\Vdash^\alpha \phi \quad \mbox{ iff }\quad \alpha\leq \val{\phi}(a),\]
	\[\mathbb{M}, x\succ^\alpha \phi \quad \mbox{ iff }\quad \alpha\leq \descr{\phi}(x).\]
	This can be equivalently expressed by means of the following  recursive definition:
	\begin{center}
		\begin{tabular}{r c l}
			$\mathbb{M}, a\Vdash^\alpha p$ & iff & $\alpha\leq \val{\phi}(a)$;\\
			$\mathbb{M}, a\Vdash^\alpha \top$ & iff & $\alpha\leq (\top^{\mathbf{A}^A})(a)$ i.e.~always;\\
			$\mathbb{M}, a\Vdash^\alpha \bot$ & iff & $\alpha\leq (\top^{\mathbf{A}^X})^\downarrow (a) = \bigwedge_{x\in X}(\top^{\mathbf{A}^X}(x)\to I(a, x)) =  \bigwedge_{x\in X} I(a, x)$;\\
			$\mathbb{M}, a\Vdash^\alpha \phi\wedge \psi$ & iff & $\mathbb{M}, a\Vdash^\alpha \phi$ and $\mathbb{M}, a\Vdash^\alpha \psi$;\\
			$\mathbb{M}, a\Vdash^\alpha \phi\vee \psi$ & iff & $\alpha\leq (\descr{\phi}\wedge\descr{\psi})^\downarrow(a) = \bigwedge_{x\in X}(\descr{\phi}(x)\wedge\descr{\psi}(x)\to I(a, x))$;\\
			$\mathbb{M}, a\Vdash^\alpha \Box \phi$ & iff & $\alpha\leq (R^{(0)}_\Box[\descr{\phi}])(a) = \bigwedge_{x\in X}(\descr{\phi}(x)\to R_\Box(a, x))$;\\
			$\mathbb{M}, a\Vdash^\alpha \Diamond \phi$ & iff & $\alpha\leq ((R^{(0)}_\Diamond[\val{\phi}])^\downarrow)(a) = \bigwedge_{x\in X}((R^{(0)}_\Diamond[\val{\phi}])(x)\to I(a, x))$\\
		\end{tabular}
	\end{center}
	
	\begin{center}
		\begin{tabular}{r c l}
			$\mathbb{M}, x\succ^\alpha p$ & iff & $\alpha\leq \descr{\phi}(x)$;\\
			$\mathbb{M}, x\succ^\alpha \bot$ & iff & $\alpha\leq (\top^{\mathbf{A}^X})(x)$ i.e.~always;\\
			$\mathbb{M}, x\succ^\alpha \top$ & iff & $\alpha\leq (\top^{\mathbf{A}^A})^\uparrow (x) = \bigwedge_{a\in A}(\top^{\mathbf{A}^A}(a)\to I(a, x)) =  \bigwedge_{a\in A} I(a, x)$;\\
			$\mathbb{M}, x\succ^\alpha \phi\vee \psi$ & iff & $\mathbb{M}, x\succ^\alpha \phi$ and $\mathbb{M}, x\succ^\alpha \psi$;\\
			$\mathbb{M}, x\succ^\alpha \phi\wedge \psi$ & iff & $\alpha\leq (\val{\phi}\wedge\val{\psi})^\uparrow(x) = \bigwedge_{a\in A}(\val{\phi}(a)\wedge\val{\psi}(a)\to I(a, x))$;\\
			$\mathbb{M}, x\succ^\alpha \Diamond \phi$ & iff & $\alpha\leq (R^{(0)}_\Diamond[\val{\phi}])(x) = \bigwedge_{a\in A}(\val{\phi}(a)\to R_\Diamond(x, a))$;\\
			$\mathbb{M}, x\succ^\alpha \Box \phi$ & iff & $\alpha\leq ((R^{(0)}_\Box[\descr{\phi}])^\uparrow)(x) = \bigwedge_{a\in A}((R^{(0)}_\Box[\descr{\phi}])(a)\to I(a, x))$.\\
		\end{tabular}
	\end{center}
\end{definition}

\begin{definition}\label{def:Sequent:True:In:Model} 
	A sequent $\phi \vdash \psi$ is \emph{true in model} $\mathbb{M} = (\mathbb{F}, V)$, notation  $\mathbb{M} \models \phi \vdash \psi$, if $\val{\phi} \subseteq \val{\psi}$, or equivalently, $\descr{\psi} \subseteq \descr{\phi}$.  
	A sequent $\phi \vdash \psi$ is \emph{valid} in an enriched formal $\mathbf{A}$-context $\mathbb{F}$, notation  $\mathbb{F} \models \phi \vdash \psi$, if   $\phi \vdash \psi$ is true in every model $\mathbb{M}$ based on $\mathbb{F}$.
%
\end{definition}

\section{Discussion: multi-market competition}
\label{sec:case study}
The literature on industrial behaviour has identified {\em multi-market contact} between firms as one of the possible reasons for `the edge of competition' becoming `blunt' in a focal market \cite{edwards1955conglomerate}. When two firms or more are active in many different product markets, they will not only compete in any given market with the other companies in that market, but will have a  specific competitive relation among each other which spans across the different markets in which they are active.  
In this situation, anticompetitive outcomes can ensue in the form of {\em mutual forebearance} (that is, the attitude by which `if you do not attack me in market $x$, I will not attack you in market $y$'), which has been studied game-theoretically in a two-firm/two-market environment \cite{bernheim1990multimarket}. 
However,  to understand the extent and nature of competition between real-life firms such as Unilever, l'Oreal, Nestl\'e and Yamaha, which are active in hundreds of product markets,\footnote{Unilever's  products include food and beverages (about 40 per cent of its revenue), cleaning agents, beauty products, and personal care products. l'Oreal focuses on  hair colour, skin care, sun protection, make-up, perfume, and hair care. Yamaha's products  include musical instruments (pianos, "silent" pianos, drums, guitars, brass instruments, woodwinds, violins, violas, cellos, and vibraphones), as well as  semiconductors, audio/visual, computer related products, sporting goods, home appliances, specialty metals and industrial robots.
Yamaha made the first commercially successful digital synthesizer.} and with very diverse, graded,  and mostly non symmetric competitive relationships, an even broader and more encompassing theoretical perspective is needed, and  the state of the art has not changed much since  Bernheim and Whinston's observation \cite{bernheim1990multimarket} that ``the existing literature contains virtually no formal theoretical analyses.''

Besides multimarket contact, {\em strategic similarity} is another concept that has been proposed to describe and analyse the relations between firms spanning across multiple product markets, and a factor which tends to sharpen, rather than blunt, the edge of competition \cite{gimeno1996hypercompetition}. 

In the present section, we discuss how  many-valued polarity-based models can serve as a formal environment to analyse these and other aspects of multi-market competition which we argue to be both salient and not yet investigated in the literature.

Indeed, the arena of multi-market competition can be represented as an $\mathbf{A}$-polarity $\mathbb{P} = (A, X, I)$ such that $A$ is the given set of firms under consideration, $X$ is the relevant set of product markets, and $I(a, x)\in \mathbf{A}$ encodes the extent to which firm $a$ is active in product-market $x$ for every $a\in A$ and $x\in X$. Being {\em active} in a market can be interpreted in many different ways, including the following:
\begin{enumerate}[noitemsep,topsep=0pt,parsep=0pt,partopsep=0pt]
\item  which percentage of the value of sales in the market $x$ is generated by $a$;
\item which percentage of sales revenues of $a$ is generated in the market $x$;
\item which percentage of new product introductions in market $x$ over the past year are produced by $a$.
\end{enumerate}
The list can of course go on, but  each of these possible interpretations gives rise to a collection of  (fuzzy) categories of firms and product markets (sometimes referred to below as firm/market categories),  ordered in the concept lattice hierarchy $\mathbb{P}^+$ arising from $\mathbb{P}$. This representation provides the first layer of structure for representing and analysing the arena of multi-market competition. For instance, one can gauge information on a focal firm $a$ in the context of the arena of multi-market competition in terms of the position of the category generated by $a$ within the concept lattice $\mathbb{P}^+$. That is, representing $a\in A$ as the characteristic function\footnote{That is, for any $b\in A$, $f_a (b) = 1$ if $b = a$ and $f_a (b) = 0$ otherwise.} $f_a: A\to \mathbf{A}$, the category generated by $a$ can be represented as the formal $\mathbf{A}$-concept $c_a: = (f_a^{\uparrow\downarrow}, f_a^{\uparrow})$, where $f_a^{\uparrow}: X\to \mathbf{A}$ is defined by the assignment $x\mapsto \bigwedge_{b\in A} f_a(b)\to I(b, x) = 1\to I(a, x) =I(a, x)$, that is, for every product market $x$, the value of $f_a^{\uparrow}(x)$ represents the extent to which $a$ is active in $x$,  and $f_a^{\uparrow\downarrow}: A\to \mathbf{A}$ is defined by the assignment $b\mapsto \bigwedge_{x\in X} f_a^{\uparrow}(x)\to I(b, x) = \bigwedge_{x\in X} I(a, x)\to I(b, x)$. Hence, one way for a firm $b$ to have the highest degree of membership in the category generated by the focal firm $a$ (that is, $f_a^{\uparrow\downarrow} (b) = 1$) is to be at least as active as $a$ in each product market $x$. In the case in which $f_a^{\uparrow\downarrow} (b) < 1$, the value of $f_a^{\uparrow\downarrow} (b)$ represents the greatest extent to which $b$ is less active than $a$ in any given market.   

Likewise, representing $x\in X$ as the characteristic function $u_x: X\to \mathbf{A}$, the category generated by $x$ can be represented as the formal $\mathbf{A}$-concept $c_x: = (u_x^{\downarrow}, u_x^{\downarrow\uparrow})$, where $u_x^{\downarrow}: A\to \mathbf{A}$ is defined by the assignment $a\mapsto \bigwedge_{y\in X} u_x(y)\to I(a, y) = 1\to I(a, x) =I(a, x)$, that is, for every firm $a$, the value of $u_x^{\downarrow}(a)$ represents the extent to which $a$ is active in $x$,  and $u_x^{\downarrow\uparrow}: X\to \mathbf{A}$ is defined by the assignment $y\mapsto \bigwedge_{a\in A} u_x^{\downarrow}(a)\to I(a, y) = \bigwedge_{a\in A} I(a, x)\to I(a, y)$. 
Hence, one way for a market $y$ to have the highest degree of membership in the category generated by the focal market $x$ (that is, $u_x^{\downarrow\uparrow} (y) = 1$) is if any firm $a$ is at least as active  $y$ as $a$ is in $x$. 

A third natural class of firm/market categories  arises in connection with ``baskets of products'' typically consumed by e.g.~certain demographic groups. For instance,  university students in the UK spend 2\% on toothpaste, 4\% on bread, 3\% on ice cream. 
Each such basket corresponds to a fuzzy subset $Y\subseteq X$, which can be represented again as a characteristic function $u_Y: X\to \mathbf{A}$ which gives rise, analogously to what has been discussed above, to the category $c_Y: = (u_Y^{\downarrow}, u_Y^{\downarrow\uparrow})$. Here $u_Y^{\downarrow}$ is the fuzzy category of producers catering to UK students, with a higher degree of membership indicating a greater alignment of a producer's market offerings to the spending habits of students.


These considerations bring us naturally to another relevant aspect of the analysis of multi-market competition, namely the {\em consumer groups} targeted by firms. This information can be 
encoded in  $\mathbf{A}$-relations $R_{\Box}: A\times X\to \mathbf{A}$ or to $R_{\Diamond}: X\times A\to \mathbf{A}$, each of which provides the interpretation of a $\Box$-type or of a $\Diamond$-type modal operator on the lattice $\mathbb{P}^+$, as appropriate. Possible interpretations for $R_{\Box}$- and $R_{\Diamond}$-type relations are refinements of items 1 to 3, above, with respect to given target groups of consumers. Thereby, for example, $R_{\Box}(a,x)$ could represent:
\begin{itemize}[noitemsep,topsep=0pt,parsep=0pt,partopsep=0pt]
	\item[$1'$.] the percentage of sales (by value) to a given target group $p$ in the market $x$, generated by $a$; or
	\item[$2'$.] the percentage of $a$'s revenue generated by sales to a given target group $p$, generated in market $x$;
\end{itemize}
while $R_{\Diamond}(a,x)$ could be made to encode 
\begin{itemize}[noitemsep,topsep=0pt,parsep=0pt,partopsep=0pt]
	\item[$3'$.] the percentage of new product introductions in market $x$ targeted at $p$ over the past year that are produced by $a$.	
\end{itemize}	
Hence, under the latter interpretation,  letting $c_a$ denote the category generated by firm $a$, the degree of membership of any firm $b$ in   $\Box_p c_a$ is computed as
$$\val{\Box_p c_a}(b) = \bigwedge_{x\in X} f_a^{\uparrow}(x)\to R_{\Box}(b, x) = \bigwedge_{x\in X} I(a, x)\to R_{\Box}(b, x).$$
Therefore, one way for a firm $b$ to have the highest degree of membership in $\Box_p c_a$  (that is, $\val{\Box_p c_a}(b) = 1$) is for $b$ to be, in each product market $x$,  at least as active relative to the target group $p$  as $a$ is  in an unrelativized way.

The extent to which firms are {\em strategically similar} can be encoded in $\mathbf{A}$-relations of type $R_{\rhd}: A\times A\to \mathbf{A}$, and again, the notion of strategic similarity, formalized as $R_{\rhd}(a, b)$, can be interpreted in different ways, which includes, but is certainly not limited to:
\begin{itemize}[noitemsep,topsep=0pt,parsep=0pt,partopsep=0pt]
\item[4.] the extent to which $a$ and $b$ are active in each market;
\item[5.] the extent to which the target consumer group in which the firm $a$ generates the largest proportion of its sales revenues for firm $a$ overlaps with that of firm $b$.
\end{itemize}
Each of these interpretations provides a meaningful and precise expression of the extent to which firms are in competition.
As also observed in other contexts (cf.~\cite{graph-based-wollic, eusflat, socio-political}), these similarity relations do not need to be symmetric or transitive, in general. 
Under each interpretation, each $\mathbf{A}$-relation $R_{{\rhd}}$ can then be used to interpret a unary modal operator ${\rhd}$ on formal concepts. For instance,
if $\phi$ is a category of firms, then the degree of membership of any firm $b$ in   ${\rhd} \phi$ is computed as
$$\val{{\rhd} \phi}(b) = \bigwedge_{b'\in A} \val{\phi}(b')\to R_{\rhd}(b', b).$$ 
Thus, for example, if $\phi$ represented processed food producers, and $b$ represented Unilever, then $\val{{\rhd} \phi}(b)$ would be the minimum strategic similarity to Unilever among processed food producers.   
%
Alternatively, if we were to start from a category generated by a single product market, say category $c_x$ generated by market $x$, then the degree of membership of any firm $b$ in   ${\rhd} c_x$ is computed as
$$\val{{\rhd} c_x}(b) = \bigwedge_{b'\in A} u_x^{\downarrow}(b')\to R_{\rhd}(b', b) =  \bigwedge_{b'\in A}  I(b', x)\to R_{\rhd}(b', b).$$
Thus $b$ would have a high degree of membership to $\val{{\rhd} c_x}$ if every firm that is highly active in market $x$ had a high degree of similarity to $b$, in other words, if $b$ were a \emph{strategically typical} producer of products in $x$. 


One aspect which---to our knowledge---has not yet been investigated concerns the various ways in which product markets are {\em connected} --- for instance,  in the sense that the same target group of consumers (e.g.~teenagers) buys different products in different markets (e.g.~soft drinks and music products and instruments). The extent to which markets are connected (in respect to given groups of consumers) has important  consequences on the competition playing out across them. This notion can be encoded in $\mathbf{A}$-relations of type $R_{\lhd}: X\times X\to \mathbf{A}$. Possible interpretations of the notion of market connection, represented as $R_{\lhd}(x, y)$, are:
\begin{itemize}[noitemsep,topsep=0pt,parsep=0pt,partopsep=0pt]
\item[6.] how many firms  which are active in $x$ are also active in  $y$;
\item[7.] the extent to which the most important consumer group for market $x$ overlaps with the most important consumer group for market $y$.
\end{itemize}
Under each interpretation, each $\mathbf{A}$-relation $R_{{\lhd}}$ can then be used to interpret a unary modal operator ${\lhd}$ on formal concepts. For instance,
if $c_x$ is the category generated by focal market $x$, then the degree of membership of any market $y$ in   ${\lhd} c_x$ is computed as
$$\descr{{\lhd} c_x}(y) = \bigwedge_{z\in X} u_x^{\downarrow\uparrow}(z)\to R_{\lhd}(z, y), $$ 
and if $c_a$ is the category generated by firm $a$, then the degree of membership of any market $y$ in   ${\lhd} c_a$ is computed as
$$\descr{{\lhd} c_a}(y) = \bigwedge_{z\in X} f_a^{\uparrow}(z)\to R_{\lhd}(z, y) =  \bigwedge_{z\in X}  I(a, z)\to R_{\lhd}(z, y),$$
thus assigning a high degree of membership to a product market $y$ if all product markets in which firm $a$ is highly active are highly similar to $y$, i.e., if $y$ is a typical product market for $a$ to be active in. 

The discussion so far concretely illustrates how the present framework can serve as a formal environment where a theory of multi-market competition can be systematically developed, so that  questions can be formulated explicitly enough to support further empirical investigation. One such basic question is: what does it mean for a focal firm $a$ to have a {\em main competitor}? This question is by no means trivial for multi-product firms. For instance, although traditionally the Coca-Cola company has been Pepsico's main competitor, nowadays, because of Pepsico's presence in the food sector, companies such as Kraft, Heinz, Unilever or Nestl\'e might in fact score higher. 

%
%
%
%
%
%
Another basic question us about what it means for a
firm $a$ to have  a {\em dominant position}\footnote{A dominant position is `a position of economic strength enjoyed by an undertaking which enables it to prevent effective competition being maintained in the relevant market by giving it the power to behave to an appreciable extent independently of its competitors, customers and ultimately of consumers’ (Case 27/76, United Brands).} in a given product market $x$. For instance, if the focal firm $a$ has only  30\% market share in market $x$, but all its multi-market competitors in $x$, as far as they are involved in other markets, serve for 80\% consumers from a given consumer group $p$,  and $a$ has  65\% share of that particular consumer group over all markets, then this will rather strengthen $a$'s competitive position vis-a-vis its competitors, because $a$ can use its strong position in $p$ to threaten the competitive positions of its competitors in other product markets, even those where $a$ is not yet active, because $a$ can easily target the relevant consumers from $p$, given how dominant $a$ is in $p$. 

So, as a hypothetical example, if Unilever has a serious market share in the market $x$ of ice-creams, but all its  competitors  in $x$, in as far as they are active in other product  markets, are largely targeting the teenagers consumer group $p$, and  Unilever has  50\% share of over all markets in group $p$, this makes Unilever's position in $x$ more dominant than it would prima facie look.

Similar considerations apply when evaluating not only the  (dominant) position of a firm {\em statically}, i.e.~in the present moment, but also {\em dynamically}, i.e.~how it might evolve in the future: for instance, 
%
compared to e.g.~Unilever, Pepsico is much stronger in the children/teenagers consumer group; so if Pepsico were planning a major acquisition in a market such as potato chips,  which is also heavily for teenagers/children,
(e.g.~Pepsico tries to buy Kettle Chips from present owner Campbell Soup Company of Andy Warhol fame),  
this acquisition would be more dangerous than the market share in potato chips would convey because of Pepsico's overall strong position in children food products.

\section{Conclusions}

In this paper, we have introduced a complete many-valued semantic environment for (multi) modal languages based on the logic of general  (i.e.~not necessarily distributive) lattices, and  have illustrated its potential as a tool for the formal analysis of multi-market competition. As this is only a preliminary exploration, many questions arise, both technical and conceptual, of which here we list a few.

\paragraph{Empirical investigation. } 

We proposed new ways to understand strategic similarity which can be quantified and provide precise measures of the extent to which firms are in competition with each other. This offers a new foundation for empirical studies into the antecedents and especially the consequences of two or more firms being in competition, exploring, for instance, the effect of the competitive intensity between the firms on the likelihood of them investing in R\&D, or introducing radical innovations. 

\paragraph{Different classifications. } 

The possible contributions above build on particular classification systems of product markets and consumer groups, as is usual in most studies in industrial economics or strategic management. However, our approach lends itself well to exploring the consequences of the possibility to categorize both markets and consumer groups in many different ways. Categorizing markets on the basis of feature-based categories, such as fruit or bicycles, creates a different view of competitive processes than using goal-based categories, such as products for leisure or status-symbol products. There might well be situations in which more than one classification system is applicable - in the sense of having influence on evaluatory decisions of consumers - at the same time. Our approach again allows to quantify who competes with whom and to which extent in all relevant classification systems, and propose composite measures of strategic similarity given these relevant classification systems \cite{wijnberg2011classification}.  

\paragraph{Dynamics. }

Our approach also offers a new foundation for new theoretical and empirical exploration of dynamics. The competitive dynamics in the sense of changing degrees to which firms are strategically similar and compete with each other in the relevant product markets and consumer groups,  the classificatory dynamics of the categories of product markets and consumer groups themselves, and the interrelations between both types of dynamics, which is at the heart of a true understanding of innovatory dynamics in modern market competition (see also \cite{wijnberg2011classification}).

Wijnberg, Nachoem M. "Classification systems and selection systems: The risks of radical innovation and category spanning." Scandinavian Journal of Management 27, no. 3 (2011): 297-306.

\paragraph{More expressive languages.}  We conjecture that the proof of completeness of the logic of  Section \ref{sec:logics} given in Appendix \ref{sec:completeness} can be extended modularly to more expressive languages  that display essentially ``many valued" features in analogy with those considered in \cite{bou2011minimum}. This is current work in progress.

\paragraph{Sahlqvist theory for many-valued non-distributive logics.} A natural direction of research is to develop the generalized Sahlqvist theory for the logics of graph-based $\mathbf{A}$-frames, by extending the results of  \cite{CMR} on Sahlqvist theory for  many-valued logics on a distributive base.

\bibliographystyle{plain}
\bibliography{BIBeusflat2019-2}

\appendix

\section{Completeness}
\label{sec:completeness}
\newcommand{\PP}{\mathbb{Fm}}
\newcommand{\SD}{\mathbb{Fm}}
\newcommand{\SI}{\mathbb{Fm}}

This section is an adaptation of  the completeness result of \cite[Appendix A]{eusflat} and \cite[Appendix B]{socio-political}. 


For any lattice $\mathbb{L}$, an $\mathbf{A}$-{\em filter} is an $\mathbf{A}$-subset of $\mathbb{L}$, i.e.~a map $f:\mathbb{L}\to \mathbf{A}$, which is both $\wedge$- and $\top$-preserving, i.e.~$f(\top) = 1$ and $f(a\wedge b) = f(a)\wedge f(b)$ for any $a, b\in\mathbb{L}$. Intuitively, the $\wedge$-preservation encodes a many-valued version of closure under $\wedge$ of filters. An $\mathbf{A}$-filter is {\em proper} if it is also $\bot$-preserving, i.e.~$f(\bot)= 0$.   Dually, an $\mathbf{A}$-{\em ideal} is a map $i:\mathbb{L}\to \mathbf{A}$ which is both $\vee$- and $\bot$-reversing, i.e.~$i(\bot) = \top$ and $i(a\vee b) = i(a)\wedge i(b)$ for any $a, b\in\mathbb{L}$, and is {\em proper} if in addition $i(\top) = 0$. 
We let $\mathsf{F}_{\mathbf{A}}(\mathbb{L})$ and $\mathsf{I}_{\mathbf{A}}(\mathbb{L})$
 respectively denote the set of  proper $\mathbf{A}$-filters and  proper $\mathbf{A}$-ideals
of $\mathbb{L}$. 
For any $\mathcal{L}$-algebra $(\mathbb{L}, \Box, \Diamond)$, and any $\mathbf{A}$-subset $k:\mathbb{L}\to \mathbf{A}$, let $k^{-\Diamond}: \mathbb{L}\to \mathbf{A}$  be defined as $k^{-\Diamond}(a)=\bigvee\{k(b)\mid \Diamond b\leq a\}$
and let $k^{-\Box}: \mathbb{L}\to \mathbf{A}$  be defined as $k^{-\Box}(a)=\bigvee \{k(b)\mid a \leq \Box b \}$. 
By definition one can see that $k(a)\leq k^{-\Diamond}(\Diamond a)$ and 
$k(a)\leq k^{-\Box}(\Box a)$ for every $a\in \mathbb{L}$.
Let $\mathbf{Fm}$ be the Lindenbaum-Tarski algebra 
of $\mathcal{L}$-formulas. 
\begin{lemma}\label{lemma:f minus diam is filter}
\begin{enumerate}
\item If $f:\mathbb{L}\to \mathbf{A}$ is  an $\mathbf{A}$-filter, then so is $f^{-\Diamond}$.  
\item If $f:\mathbf{Fm}\to \mathbf{A}$ is  a proper $\mathbf{A}$-filter, then so is $f^{-\Diamond}$.
\item If $i:\mathbb{L}\to \mathbf{A}$ is an $\mathbf{A}$-ideal, then so is $i^{-\Box}$. 
\item If $i:\mathbf{Fm}\to \mathbf{A}$ is a proper $\mathbf{A}$-ideal, then so is $i^{-\Box}$.
\item If $\phi, \psi\in\mathcal{L}$, then $\top\vdash \phi\vee \psi$ implies that $\top\vdash\phi$ or $\top\vdash\psi$.
\item If $\phi, \psi\in\mathcal{L}$, then $\phi\not\vdash \bot$ and $\psi\not\vdash \bot$ implies that $\phi \wedge \psi\not\vdash \bot$.
\end{enumerate}
\end{lemma}  
\begin{proof} 
We only prove items 3 and 4, as the other items were proven in \cite[Lemma A.1]{eusflat}. For 3, we first show that $i^{-\Box}$ is $\bot$-reversing:
\begin{center}
\begin{tabular}{cll}
   $i^{-\Box}(\bot)$ & = & $\bigvee\{i(b)\mid \bot\leq \Box b\}$\\
   & = & $\bigvee\{i(b)\mid b\in \mathbb{L}\}$\\
   & = & $i(\bot)$\\
   & = & $1$\\
\end{tabular}
\end{center}
We now show that $i^{-\Box}$ is $\vee$-reversing. For all $a, b \in \mathbb{L}$,
\begin{center}
\begin{tabular}{cll}
   &$i^{-\Box}(a)\wedge i^{-\Box}( b)$\\
    = &$ \bigvee\{i(c_1)\mid a \leq \Box c_1\} \wedge \bigvee\{i(c_2)\mid b \leq\Box c_2\}$\\
        = &$ \bigvee\{i(c_1)\wedge i(c_2)\mid a \leq \Box c_1 \mbox{ and } b \leq \Box c_2\}$ & (frame dist.) \\ 
             = &$ \bigvee\{i(c_1\wedge c_2)\mid a \leq \Box c_1 \mbox{ and }b \leq \Box c_2\}$ & ($i$ is an $\mathbf{A}$-ideal)\\
               $\leq$ &$ \bigvee\{f(c)\mid a \leq \Box c \mbox{ and } b \leq \Box c\}$ & ($\ast\ast$)\\
 = &$\bigvee\{f(c)\mid a \vee b \leq \Box c\}$ & \\
=&$i^{-\Box}(a\vee b)$,\\
 \end{tabular}
 \end{center}
 where the inequality marked with ($\ast\ast$) follows from the fact that $\Box c_1\vee \Box c_2\leq \Box(c_1\vee c_2)$. For the converse inequality, observe that $\bigvee\{i(c) \mid a \leq \Box c \mbox{ and } b \leq \Box c\} \leq \bigvee\{i(c) \mid a \leq \Box c\}$ and $\bigvee\{i(c) \mid a \leq \Box c \mbox{ and } b \leq \Box c\} \leq \bigvee\{i(c) \mid b \leq \Box c\}$, which means
\[
\bigvee\{i(c) \mid a \leq \Box c \mbox{ and } b \leq \Box c\} \leq \bigvee\{i(c) \mid a \leq \Box c\} \wedge \bigvee\{i(c) \mid b \leq \Box c\}. 
\]
Hence, 
\begin{center}
\begin{tabular}{cll}
   &$i^{-\Box}(a \vee b)$\\
    = &$ \bigvee\{i(c)\mid a \vee b\leq \Box c\}$\\
        = &$ \bigvee\{i(c)\mid a \leq \Box c \mbox{ and } b \leq \Box c\}$ & \\ 
             $\leq$ &$ \bigvee\{i(c)\mid a \leq \Box c\} \wedge \bigvee\{i(c)\mid b \leq \Box c\}$ & \\
	= &$i^{-\Box}(a) \wedge i^{-\Box}(b)$.\\
 \end{tabular}
 \end{center}

\noindent We now prove item 4. In the algebra $\mathbf{Fm}$, $i^{-\Box}([\top]) = \bigvee\{i([\phi])\mid [\top] \leq [\Box \phi]\} = \bigvee\{i([\phi])\mid \top \vdash \Box \phi\} = \bigvee\{i([\phi])\mid \top \vdash \phi\} = i([\top]) = 0$. The crucial equality is the third to last, which holds since  $\top \vdash \Box \phi$ iff $\top \vdash \phi$. The right to left implication can be easily derived in 
$\mathbf{L}$. 
For the sake of the left-to-right implication we appeal to the completeness of 
$\mathbf{L}$ 
with respect to the class of all normal lattice expansions of the appropriate signature \cite{CoPa:non-dist} and reason contrapositively. Suppose $\top \not\vdash \phi$. Then, by this completeness theorem, there is a lattice expansion $\mathbb{C}$ and assignment $v$ on $\mathbb{C}$ such that $v(\phi) \neq 1$. Now consider the algebra $\mathbb{C}'$ obtained from $\mathbb{C}$ by adding a new top element $1'$ and extending the $\Box$-operation by declaring $\Box 1' = 1'$. We keep the assignment $v$ unchanged. It is easy to check that $\mathbb{C}'$ is a normal lattice expansion, and that $v(\Box \phi) \leq 1 < 1'$ and hence $\top \not \vdash \Box \phi$.
\end{proof}


\begin{lemma}\label{eq:premagicnew2} 
For any $f\in \mathsf{F}_{\mathbf{A}}(\mathbb{L})$ and any $i\in \mathsf{I}_{\mathbf{A}}(\mathbb{L})$,
\begin{enumerate}
\item $\bigvee_{b\in \mathbb{L}} (f^{-\Diamond}(b)\otimes i(b)) = \bigvee_{a\in \mathbb{L}} (f(a)\otimes i(\Diamond a))$;
\item $\bigvee_{b\in \mathbb{L}} (f(b) \otimes i^{-\Box}(b)) = \bigvee_{a\in \mathbb{L}} (f(\Box a)\otimes i(a))$. 
\end{enumerate}
\end{lemma}
\begin{proof} For the right-to-left inequality of (1) we use the fact that
$f(a)\leq f^{-\Diamond}(\Diamond a)$ implies that  
$f(a)\otimes i(\Diamond a)\leq f^{-\Diamond}(\Diamond a)\otimes i(\Diamond a)$ for every $a\in \mathbb{L}$, which gives $\bigvee_{a\in \mathbb{L}} (f(a)\otimes i(\Diamond a))\leq\bigvee_{b\in \mathbb{L}} (f^{-\Diamond}(b)\otimes i(b)).$ Conversely, to show that \[\bigvee_{b\in \mathbb{L}} (f^{-\Diamond}(b)\otimes i(b))\leq\bigvee_{a\in \mathbb{L}} (f(a)\otimes i(\Diamond a)),\] 
it is enough to show that, for every $b\in \mathbb{L}$,
 \[f^{-\Diamond}(b)\otimes i(b)\leq\bigvee_{a\in \mathbb{L}} (f(a)\otimes i(\Diamond a)),\] 
 or equivalently, by the definition of $f^{-\Diamond}(b)$ and the fact that $\otimes$ is completely join-preserving in its first coordinate,
  \[\bigvee_{\Diamond c\leq b} (f(c)\otimes i(b))\leq\bigvee_{a\in \mathbb{L}} (f(a)\otimes i(\Diamond a)).\] 
  Hence, let $c\in \mathbb{L}$ such that $\Diamond c\leq b$, and let us show that   \[f(c)\otimes i(b)\leq\bigvee_{a\in \mathbb{L}} (f(a)\otimes i(\Diamond a)).\] 
 Since $i$ is $\vee$-reversing, hence order-reversing, $\Diamond c\leq b$ implies $i(b)\leq i(\Diamond c)$. Therefore, 
 \[f(c)\otimes i(b)\leq f(c)\otimes i(\Diamond c)\leq  \bigvee_{a\in \mathbb{L}} (f(a)\otimes i(\Diamond a)),\]
 as required. 

For (2), we use the fact that $i(a)\leq i^{-\Box}(\Box a)$ implies $f(\Box a)\otimes i(a)\leq f(\Box a) \otimes i^{-\Box}(\Box a))$ for all $a \in \mathbb{L}$, from which we obtain 
$ \bigvee_{a\in \mathbb{L}} (f(\Box a)\otimes i(a))\leq \bigvee_{b\in \mathbb{L}} (f(b) \otimes u^{-\Box}(b))$. 
To show that 
\[\bigvee_{b\in \mathbb{L}} (f(b) \otimes i^{-\Box}(b))\leq \bigvee_{a\in \mathbb{L}} (f(\Box a)\otimes i(a))\] 
we can show that for any $b \in \mathbb{L}$ 
\[
f(b) \otimes i^{-\Box}(b)\leq \bigvee_{a\in \mathbb{L}} (f(\Box a)\otimes i(a)). 
\]   
After applying the definition of $i^{-\Box}(b)$ and the fact that $\otimes$ is completely join-preserving in its second coordinate, we obtain the equivalent inequality 
\[
\bigvee_{b \leq \Box c} ( f(b) \otimes i(c))\leq \bigvee_{a\in \mathbb{L}} (f(\Box a)\otimes i(a)). 
\]
Let $c \in \mathbb{L}$ with $b \leq \Box c$. Since $f$ is order-preserving we get 
\[ 
f(b) \otimes i(c)\leq f(\Box c) \otimes i(c) \leq\bigvee_{a\in \mathbb{L}} (f(\Box a)\otimes i(a)). 
\]
\end{proof}
\begin{definition}
\label{def:canonical frame}
Let $\mathbf{Fm}$ be the Lindenbaum--Tarski algebra  of $\mathcal{L}$-formulas.\footnote{In the remainder of this section, we abuse notation and identify formulas with their equivalence class in $\mathbf{Fm}$.}	The {\em canonical} polarity-based $\mathbf{A}$-{\em frame} is the structure  $\mathbb{P} = (\mathsf{F}_{\mathbf{A}}(\mathbf{Fm}), \mathsf{I}_{\mathbf{A}}(\mathbf{Fm}), I, R_{\Diamond}, R_\Box)$  defined as follows:\footnote{Recall that for any set $W$, the $\mathbf{A}$-{\em subsethood} relation between elements of $\mathbf{A}$-subsets of $W$ is the map $S_W:\mathbf{A}^W\times \mathbf{A}^W\to \mathbf{A}$ defined as $S_W(f, g) :=\bigwedge_{w\in W}(f(w)\rightarrow g(w)) $. 
If $S_W(f, g) =1$ we also write $f\subseteq g$. 
}
 
$I: \mathsf{F}_{\mathbf{A}}(\mathbf{Fm})\times \mathsf{I}_{\mathbf{A}}(\mathbf{Fm})\to \mathbf{A}$, $R_\Diamond : \mathsf{I}_{\mathbf{A}}(\mathbf{Fm})\times\mathsf{F}_{\mathbf{A}}(\mathbf{Fm})\to \mathbf{A}$ and $R_{\Box}: \mathsf{F}_{\mathbf{A}}(\mathbf{Fm})\times \mathsf{I}_{\mathbf{A}}(\mathbf{Fm}) \to \mathbf{A}$ 
	are  defined as follows: \[I(f, i): = \bigvee_{\phi\in \mathbf{Fm}}(f(\phi)\otimes i(\phi));\]
	 \[R_\Diamond(i, f) := \bigvee_{\phi\in \mathbf{Fm}}(f^{-\Diamond}(\phi)\otimes  i(\phi)) = \bigvee_{\phi\in \mathbf{Fm}}(f(\phi)\otimes  i(\Diamond\phi));\]
\[
R_{\Box}(f,i):= \bigvee_{\phi \in \mathbf{Fm}} (f(\phi) \otimes i^{-\Box}(\phi)) = \bigvee_{\phi \in \mathbf{Fm}}
(f(\Box \phi)\otimes i(\phi)).
\]
	\end{definition}
	\begin{lemma}
	The structure $\mathbb{P}$ of Definition \ref{def:canonical frame}
 is a polarity-based $\mathbf{A}$-frame.	\end{lemma}
	\begin{proof}
We need to show that $R_{\Diamond}$ is $I$-compatible, i.e.,
\begin{align*}
(R_\Diamond^{(0)}[\{\beta / f \}])^{\downarrow\uparrow} &\subseteq R_\Diamond^{(0)}[\{\beta /  f \}] \\
(R_\Diamond^{(1)}[\{\beta /  i \}])^{\uparrow\downarrow} &\subseteq R_\Diamond^{(1)}[\{\beta / i \}],
\end{align*}
and that $R_{\Box}$ is $I$-compatible, i.e.,
\begin{align*}
(R_\Box^{(0)}[\{\beta /  i \}])^{\uparrow\downarrow} &\subseteq R_\Box^{(0)}[\{\beta / i \}] \\
(R_\Box^{(1)}[\{\beta /  f \}])^{\downarrow\uparrow} &\subseteq R_\Box^{(1)}[\{\beta /  f \}].
\end{align*}
Considering the second inclusion for $R_{\Diamond}$, by definition, for any $i\in\mathsf{I}_{\mathbf{A}}(\mathbf{Fm})$,
\begin{center}
\begin{tabular}{c l}
&$R_{\Diamond}^{(1)}[\{\beta / i \}](f)$ \\
 = & $\bigwedge_{i'\in \mathsf{I}_{\mathbf{A}}(\mathbf{Fm})}[\{\beta / i \}(i')\to  R_{\Diamond}(i',f)]$\\
  = & $\beta\to R_{\Diamond}(i, f)$\\
\end{tabular}
\end{center}
and
\begin{center}
\begin{tabular}{c l}
 & $(R_\Diamond^{(1)}[\{\beta /  i \}])^{\uparrow\downarrow} (f)$\\
 = & $\bigwedge_{i'\in \mathsf{I}_{\mathbf{A}}(\mathbf{Fm})}[(R_\Diamond^{(1)}[\{\beta /  i \}])^{\uparrow}(i')\to I(f, i') ],$ \\
\end{tabular}
\end{center}
and hence it is enough to find some $i'\in \mathsf{I}_{\mathbf{A}}(\mathbf{Fm})$ such that 
\[(R_\Diamond^{(1)}[\{\beta /  i \}])^{\uparrow}(i')\to I(f, i')\leq \beta\to R_{\Diamond}(i, f),\]
i.e.
\begin{center}
\begin{tabular}{cl l}
&$\bigwedge_{f'\in \mathsf{F}_{\mathbf{A}}(\mathbf{Fm})}\left([\beta\to R_\Diamond(i,f')]\to I(f',i')\right)\to I(f,i')\leq \beta\to R_{\Diamond}(i,f) $ & ($\ast$)\\
\end{tabular}
\end{center}
Let $i': \mathbf{Fm}\to\mathbf{A}$ be  defined by the assignment 
\[ 
i'(\phi) =
\left\{ \begin{array}{ll} 
0 &\mbox{if $\top \vdash \phi$} \\
i(\Diamond\phi) &\mbox{otherwise.}
\end{array} \right.
\]
Using Lemma \ref{lemma:f minus diam is filter} it can be readily verified that $i'$ is an $\mathbf{A}$-ideal. 
Moreover, 
\begin{center}
	\begin{tabular}{rcl}
		$I(f, i')$ & = & $\bigvee_{\phi\in \mathbf{Fm}} (f(\phi)\otimes i'(\phi))$\\
		& = & $\bigvee_{\phi\in \mathbf{Fm}} (f(\phi)\otimes i(\Diamond \phi))$\\
		& = & $\bigvee_{\phi\in \mathbf{Fm}} (f^{-\Diamond}(\phi)\otimes i(\phi))$\\
		& = & $R_{\Diamond}(i, f)$,\\
	\end{tabular}
\end{center}
and likewise $I(f', i') = R_{\Diamond}(i, f')$. Therefore, for this choice of $i'$, inequality ($\ast$) can be rewritten as follows:
\begin{center}
	\begin{tabular}{cl l}
		&$\bigwedge_{f'\in \mathsf{F}_{\mathbf{A}}(\mathbf{Fm})}\left([\beta\to R_\Diamond(i,f')]\to R_{\Diamond}(i,f')\right)\to R_{\Diamond}(i,f)\leq \beta\to R_{\Diamond}(i,f) $ & \\
	\end{tabular}
\end{center}
The inequality above is true if \[\beta \leq \bigwedge_{f'\in \mathsf{F}_{\mathbf{A}}(\mathbf{Fm})}\left([\beta\to R_\Diamond(i,f')]\to R_{\Diamond}(i,f')\right),\]
i.e.~if for every $f'\in \mathsf{F}_{\mathbf{A}}(\mathbf{Fm})$, 
\[\beta \leq [\beta\to R_\Diamond(i,f')]\to R_{\Diamond}(i,f'),\]
which is an instance of a tautology in residuated lattices.

Let us show that $(R_\Diamond^{(0)}[\{\beta / f \}])^{\downarrow\uparrow} \subseteq R_\Diamond^{(0)}[\{\beta / f \}]$. By definition, for every $f\in \mathsf{I}_{\mathbf{A}}(\mathbf{Fm})$,
\begin{center}
\begin{tabular}{cl}
& $R_\Diamond^{(0)}[\{\beta /  f\}] (i)$\\
= &$\bigwedge_{f'\in \mathsf{F}_{\mathbf{A}}(\mathbf{Fm})}[\{\beta /  f\}(f')\to R_{\Diamond}(i,f')]$\\
= &$\beta\to R_{\Diamond}(i,f)$\\
\end{tabular}
\end{center}
and
\begin{center}
\begin{tabular}{c l}
& $(R_\Diamond^{(0)}[\{\beta / f \}])^{\downarrow\uparrow}(i)$\\
= &$\bigwedge_{f'\in \mathsf{F}_{\mathbf{A}}(\mathbf{Fm})}[(R_\Diamond^{(0)}[\{\beta / f \}])^{\downarrow}(f')\to I(f', i)]$.\\
\end{tabular}
\end{center}
Hence, it is enough to find some $f'\in \mathsf{F}_{\mathbf{A}}(\mathbf{Fm})$ such that 
\[(R_\Diamond^{(0)}[\{\beta / f \}])^{\downarrow}(f')\to I(f', i)\leq \beta\to R_{\Diamond}(i, f),\]
i.e.
\begin{center}
\begin{tabular}{cll}
& $\bigwedge_{i'\in \mathsf{I}_{\mathbf{A}}(\mathbf{Fm})}\left([\beta\to R_{\Diamond}(i', f)]\to I(f', i')\right)\to I(f', i)\leq \beta\to R_{\Diamond}(i, f)$ & \\
\end{tabular}
\end{center}
Let $f'\in \mathsf{F}_{\mathbf{A}}(\mathbf{Fm})$ such that $f' = f^{-\Diamond}$ (cf.~Lemma \ref{lemma:f minus diam is filter}). Then
\begin{center}
	\begin{tabular}{rcl}
		$I(f', i)$ & = & $\bigvee_{\phi\in \mathbf{Fm}} \left(f^{-\Diamond}(\phi)\otimes i(\phi)\right)$\\
		& = & $R_{\Diamond}(i, f)$,\\
	\end{tabular}
\end{center}
and likewise $I(f',i') = R_{\Diamond}(i', f)$. Therefore, for this choice of $f'$, 
\begin{center}
	\begin{tabular}{cll}
		& $\bigwedge_{i'\in \mathsf{I}_{\mathbf{A}}(\mathbf{Fm})}\left([\beta\to R_{\Diamond}(i', f)]\to R_{\Diamond}(i', f)\right)\to R_{\Diamond}(i, f)\leq \beta\to R_{\Diamond}(i, f)$ & \\
	\end{tabular}
\end{center}
which is shown to be true by the same argument as the one concluding the verification of  the previous inclusion.

For the second inclusion involving $R_\Box$, it is helpful to first observe that for any $f \in \mathsf{F}_{\mathbf{A}}(\mathbf{Fm})$: 
\begin{center}
\begin{tabular}{c l}
&$R_{\Box}^{(1)}[\{\beta / f \}](\alpha,w)$ \\
 = & $\bigwedge_{f'\in\mathsf{F}_{\mathbf{A}}(\mathbf{Fm})}(\{\beta / f \}(f')\to  R_{\Box}(f',i))$\\
= & $\beta\to R_{\Box}(f, i)$\\
\end{tabular}
\end{center}
and
\begin{center}
\begin{tabular}{c l}
 & $(R_\Box^{(1)}[\{\beta /  f \}])^{\downarrow\uparrow} (i)$\\
 = & $\bigwedge_{z'\in Z_A}((R_\Box^{(1)}[\{\beta / f \}])^{\downarrow}(f')\to I(f', i))$. \\
\end{tabular}
\end{center}
and hence it is enough to find some $f'\in \mathsf{F}_{\mathbf{A}}(\mathbf{Fm})$ such that 
\[(R_\Box^{(1)}[\{\beta / f \}])^{\downarrow}(f')\to I(f', i)\leq \beta\to R_{\Box}(f, i),\]
i.e.
\begin{center}
\begin{tabular}{cl l}
&$\bigwedge_{i'\in \mathsf{I}_{\mathbf{A}}(\mathbf{Fm})}\left([\beta\to R_\Box(f,i')]\to I(f',i')\right)\to I(f',i)\leq \beta\to R_{\Box}(f, i) $ & ($\ast\ast$)\\
\end{tabular}
\end{center}
Let $f': \mathbf{Fm}\to\mathbf{A}$ be  defined by the assignment 
\[ 
i'(\phi) =
\left\{ \begin{array}{ll} 
0 &\mbox{if $\phi \vdash \bot$} \\
i(\Box\phi) &\mbox{otherwise.}
\end{array} \right.
\]
Using Lemma \ref{lemma:f minus diam is filter} it can be verified that $f'$ is an $\mathbf{A}$-filter. 
Moreover, 
\begin{center}
	\begin{tabular}{rcl}
		$I(f', i)$ & = & $\bigvee_{\phi\in \mathbf{Fm}} (f'(\phi)\otimes i(\phi))$\\
		& = & $\bigvee_{\phi\in \mathbf{Fm}} (f(\Box\phi)\otimes i(\phi))$\\
		& = & $\bigvee_{\phi\in \mathbf{Fm}} (f(\phi)\otimes i^{-\Box}(\phi))$\\
		& = & $R_{\Box}(f, i)$,\\
	\end{tabular}
\end{center}
and likewise $I(f', i') = R_{\Box}(f, i')$. Therefore, for this choice of $f'$, inequality ($\ast\ast$) can be rewritten as follows:
\begin{center}
	\begin{tabular}{cl l}
		&$\bigwedge_{i'\in \mathsf{I}_{\mathbf{A}}(\mathbf{Fm})}\left([\beta\to R_\Box(f, i')]\to R_{\Box}(f, i')\right)\to R_{\Box}(f, i)\leq \beta\to R_{\Box}(f, i) $ & \\
	\end{tabular}
\end{center}
The inequality above is true if \[\beta \leq \bigwedge_{i'\in \mathsf{I}_{\mathbf{A}}(\mathbf{Fm})}\left([\beta\to R_\Box(f, i')]\to R_{\Box}(f, i')\right),\]
i.e.~if for every $i'\in \mathsf{i}_{\mathbf{A}}(\mathbf{Fm})$, 
\[\beta \leq [\beta\to R_\Diamond(f,i')]\to R_{\Diamond}(f, i'),\]
which is an instance of a tautology in residuated lattices.
%

Let us show that $(R_\Box^{(0)}[\{\beta / i \}])^{\uparrow\downarrow} \subseteq R_\Box^{(0)}[\{\beta / i\}]$. By definition, for every $i\in \mathsf{I}_{\mathbf{A}}(\mathbf{Fm})$,
\begin{center}
\begin{tabular}{cl}
& $R_\Box^{(0)}[\{\beta / i\}] (f)$\\
= &$\bigwedge_{i'\in \mathsf{I}_{\mathbf{A}}(\mathbf{Fm})}[\{\beta / i \}(i')\to R_{\Box}(f, i')]$\\
= &$\beta\to R_{\Box}(f, i)$\\
\end{tabular}
\end{center}
and
\begin{center}
\begin{tabular}{c l}
& $(R_\Box^{(0)}[\{\beta / i \}])^{\uparrow\downarrow}(f)$\\
= &$\bigwedge_{i'\in \mathsf{I}_{\mathbf{A}}(\mathbf{Fm})}[(R_\Box^{(0)}[\{\beta / i \}])^{\uparrow}(i')\to I(f, i')]$.\\
\end{tabular}
\end{center}
Hence, it is enough to find some $i'\in \mathsf{I}_{\mathbf{A}}(\mathbf{Fm})$ such that 
\[(R_\Box^{(0)}[\{\beta / i \}])^{\uparrow}(i')\to I(f, i')\leq \beta\to R_{\Box}(f, i),\]
i.e.
\begin{center}
\begin{tabular}{cll}
& $\bigwedge_{f'\in \mathsf{F}_{\mathbf{A}}(\mathbf{Fm})}\left([\beta\to R_{\Box}(f', i)]\to I(f', i')\right)\to I(f, i')\leq \beta\to R_{\Box}(f, i)$ & \\
\end{tabular}
\end{center}
Let $i'\in \mathsf{I}_{\mathbf{A}}(\mathbf{Fm})$ such that $i' = i^{-\Box}$ (cf.~Lemma \ref{lemma:f minus diam is filter}). Then
\begin{center}
	\begin{tabular}{rcl}
		$I(f, i')$ & = & $\bigvee_{\phi\in \mathbf{Fm}}\left(f(\phi)\otimes i^{-\Box}(\phi)\right)$\\
		& = & $R_{\Box}(f, i)$,\\
	\end{tabular}
\end{center}
and likewise $I(f',i') = R_{\Box}(f', i)$. Therefore, for this choice of $i'$, 
\begin{center}
	\begin{tabular}{cll}
		& $\bigwedge_{f'\in \mathsf{F}_{\mathbf{A}}(\mathbf{Fm})}\left([\beta\to R_{\Box}(f', i)]\to R_{\Box}(f', i)\right)\to R_{\Box}(f, i)\leq \beta\to R_{\Box}(f, i)$ & \\
	\end{tabular}
\end{center}
which is shown to be true by the same argument as the one concluding the verification of  the previous inclusion. 
%
\end{proof}

\begin{definition}
\label{def:canonical model}
Let $\mathbf{Fm}$ be the Lindenbaum--Tarski algebra of $\mathcal{L}$-formulas.	The {\em canonical graph-based} $\mathbf{A}$-{\em model} is the structure  $\mathbb{M} =(\mathbb{P}, V)$ such that $\mathbb{P}$ is the canonical graph-based $\mathbf{A}$-frame of Definition \ref{def:canonical frame}, and  if $p \in \mathsf{Prop}$, then $V(p) = (\val{p}, \descr{p})$ with $\val{p}: \mathsf{F}_{\mathbf{A}}(\mathbf{Fm})\to \mathbf{A}$ and $\descr{p}: \mathsf{I}_{\mathbf{A}}(\mathbf{Fm})\to \mathbf{A}$ defined by  $f\mapsto f(p)$ and  $i\mapsto i(p)$, respectively. 
\end{definition}
\begin{lemma}
	The structure $\mathbb{P}$ of Definition \ref{def:canonical model}
 is a polarity-based $\mathbf{A}$-model.
	\end{lemma}
	\begin{proof}
It is enough to show that $\val{p}^{\uparrow}=\descr{p}$ and $\val{p}=\descr{p}^{\downarrow}$  for any $p \in \mathsf{Prop}$.  To show that $\descr{p} (i)\leq \val{p}^{\uparrow}(i)$ for any $i\in \mathsf{I}_{\mathbf{A}}(\mathbf{Fm})$, by definition, we need to show that
\[i(p)\leq \bigwedge_{f\in \mathsf{F}_{\mathbf{A}}(\mathbf{Fm})}(\val{p}(f)\to I(f,i)),\]
i.e.~that for every $f\in \mathsf{F}_{\mathbf{A}}(\mathbf{Fm})$, 
\[i(p)\leq \val{p}(f)\to I(f, i).\]
By definition, the inequality above is equivalent to
\[i(p)\leq f(p)\to \bigvee_{\phi\in \mathbf{Fm}} [f(\phi)\otimes i(\phi)].\]
Since $f(p)\otimes i(p)\leq \bigvee_{\phi\in \mathbf{Fm}} [f(\phi)\otimes i(\phi)]$ 
it is enough to show that 
\[i(p)\leq f(p)\to (f(p)\otimes i(p)).\]
By residuation the inequality above is equivalent to
\[i(p)\otimes f(p) \leq f(p)\otimes i(p),\]
which is the instance of a tautology in residuated lattices.
Conversely, to show that $ \val{p}^{\uparrow}(i)\leq \descr{p} (i)$, i.e.
\[\bigwedge_{f\in \mathsf{F}_{\mathbf{A}}(\mathbf{Fm})}(\val{p}(f)\to I(f, i))\leq i(p),\]
it is enough to show that
\begin{equation}
\label{eqq}
\val{p}(f)\to I(f, i)\leq i(p)
\end{equation}
for some $f\in \mathsf{F}_{\mathbf{A}}(\mathbf{Fm})$. Let  $f_p:\mathbf{Fm}\to \mathbf{A}$ be defined by the assignment 
\[
f_p(\phi) = \left\{\begin{array}{ll}
1 & \text{if }  p\vdash \phi\\
0 & \text{otherwise. } \\
\end{array}\right.
\]
 Hence, $I(f_p, i) = \bigvee_{\phi\in \mathbf{Fm}}(f(\phi)\otimes i(\phi))=\bigvee_{p\vdash\phi}i(\phi) = i(p)$, the last equality holding since $i$ is order-reversing.  
 Therefore,  
 $\val{p}(f)\to I(f_p, i) = f_p(p)\to i(p) = 1\to i(p) = i(p)$, which shows  \eqref{eqq}.
 
By adjunction, the inequality $\descr{p} \leq \val{p}^{\uparrow}$ proven above implies that   $\val{p}\leq\descr{p}^{\downarrow}$. Hence, to show that  $\val{p}=\descr{p}^{\downarrow}$, it is enough to show $\descr{p}^{\downarrow}(f)\leq\val{p}(f)$ for every $f\in \mathsf{F}_{\mathbf{A}}(\mathbf{Fm})$, i.e. 
\[\bigwedge_{i\in \mathsf{I}_{\mathbf{A}}(\mathbf{Fm})} \descr{p}(i)\to I(f, i)\leq f(p),\]
and to show the inequality above holds, it is enough to show that 
\begin{equation}\label{eqqq} \descr{p}(i)\to I(f, i)\leq f(p)\end{equation}
for some $i\in \mathsf{I}_{\mathbf{A}}(\mathbf{Fm})$.
Let 
$i_{p}: \mathbf{Fm}\to \mathbf{A}$ be defined by the following assignment:
 \[
i_{p}(\phi) = \left\{\begin{array}{ll}
1 & \text{if } \phi\vdash p \\
0 & \text{if } \phi\not\vdash p.
\end{array}\right.
\]
By construction, $i_{p}$ is $\vee$-, $\bot$- and $\top$-reversing. Moreover, $\descr{p}(i_{p}) = i_{p}(p) = 1$, and $I(f, i_{p}) = \bigvee_{\phi\in \mathbf{Fm}}(f(\phi)\otimes i_{p}(\phi))  = \bigvee_{\phi\vdash p}(f(\phi)\otimes 1)  = \bigvee_{\phi\vdash p}f(\phi)= f(p)$, where the last equality follows from the fact that $f$ is order preserving. 
 Hence, the left-hand side of \eqref{eqqq} can be equivalently rewritten as $1\to f(p) = f(p)$, which shows \eqref{eqqq} and concludes the proof.
\end{proof}
\begin{lemma}[Truth Lemma]
\label{lemma:truthlemma}
	For every  $\phi\in \mathbf{Fm}$, the maps $\val{\phi}: \mathsf{F}_{\mathbf{A}}(\mathbf{Fm})\to \mathbf{A}$ and $\descr{\phi}: \mathsf{I}_{\mathbf{A}}(\mathbf{Fm})\to \mathbf{A}$ coincide with those defined by the assignments $f\mapsto f(\phi)$ and $u\mapsto i(\phi)$, respectively. 
\end{lemma}
\begin{proof}
We proceed by  induction on $\phi$. If $\phi: = p\in \mathsf{Prop}$, the statement follows immediately from  Definition \ref{def:canonical model}. 

If $\phi: =\top$, then $\val{\top}(f) = 1 = f(\top)$ since $\mathbf{A}$-filters are $\top$-preserving. Moreover,
\begin{center}
\begin{tabular}{r cl}
 $\descr{\top}(i)$& $=$& $\val{\top}^{\uparrow}(i)$\\
 & $=$& $ \bigwedge_{f\in \mathsf{F}_{\mathbf{A}}(\mathbf{Fm})} [\val{\top}(f)\to I(f, i)]$\\
  & $=$& $ \bigwedge_{f\in \mathsf{F}_{\mathbf{A}}(\mathbf{Fm})} [f(\top)\to I(f, i)]$\\
  & $=$& $ \bigwedge_{f\in \mathsf{F}_{\mathbf{A}}(\mathbf{Fm})} [I(f,i)]$. \\
  \end{tabular}
 \end{center}
To show that $\descr{\top}(i)\leq  i(\top)=0$, i.e.~that \[\bigwedge_{f\in \mathsf{F}_{\mathbf{A}}(\mathbf{Fm})} I(f, i)\leq 0,\]
it is enough  to find some $f\in \mathsf{F}_{\mathbf{A}}(\mathbf{Fm})$ such that  $I(f,i)=0$. Let  $f_\top:\mathbf{Fm}\to \mathbf{A}$ be defined by the assignment 
\[
f_\top(\phi) = \left\{\begin{array}{ll}
1 & \text{if }  \top\vdash \phi\\
0 & \text{otherwise. } \\
\end{array}\right.
\]\
By definition, $I(f_{\top}, i) = \bigvee_{\psi\in \mathbf{Fm}}[f_{\top}(\psi)\otimes i(\psi)] =    \bigvee_{\top\vdash \psi}i(\psi)
\leq i(\top)$, the last inequality being due to the fact that  $i$ is order-reversing. 
Hence, since $i$ is $\top$-reversing, $I(f_{\top}, i)\leq i(\top)=0$, as required.

If $\phi: =\bot$, then $\descr{\bot}(i) = 1 = i(\bot)$ since $\mathbf{A}$-ideals are $\bot$-reversing. Let us show that $\val{\bot}(f) = f(\bot)$. The inequality $f(\bot)\leq \val{\bot}(f) $ follows immediately from the fact that $f$ is proper and hence that $f(\bot) = 0$. To show that $\val{\bot}(f) \leq f(\bot)$, by definition  $\val{\bot}(f)=\descr{\bot}^{\downarrow}(f) =  \bigwedge_{i\in \mathsf{I}_{\mathbf{A}}(\mathbf{Fm})} [(i(\bot))\to I(f, i)] = \bigwedge_{i\in \mathsf{I}_{\mathbf{A}}(\mathbf{Fm})} I(f, i)$,
%
hence, it is enough to find some $i\in \mathsf{I}_{\mathbf{A}}(\mathbf{Fm})$ such that  \begin{equation}\label{eqqqqqq}I(f, i)=0.\end{equation}
Let  $i_{\bot}:\mathbf{Fm}\to \mathbf{A}$ be defined by the assignment 
	\[
i_{\bot}(\psi) = \left\{\begin{array}{ll}
1 & \text{if } \psi\vdash \bot\\
0 & \text{if }\psi\not\vdash \bot.
\end{array}\right.
\]
By definition and since $f$ is order-preserving and $\bot$-preserving, $I(f, i_{\bot}) = \bigvee_{\psi\in \mathbf{Fm}}[f(\psi)\otimes i_{\bot}(\psi)] = \bigvee_{\psi\vdash\bot}f(\psi)= f(\bot)=0$. 

If  $\phi: = \phi_1\wedge\phi_2$, then $\val{\phi_1\wedge\phi_2}(f) = (\val{\phi_1} \wedge \val{\phi_2})(f) = \val{\phi_1}(f) \wedge \val{\phi_2}(f) = f(\phi_1) \wedge f(\phi_2) = f(\phi_1 \wedge \phi_2)$. 
Let us  show that $\descr{\phi_1 \wedge \phi_2}(i) = i(\phi_1 \wedge \phi_2)$. By definition,
\begin{center}
\begin{tabular}{cl}
& $\descr{\phi_1 \wedge \phi_2}(i)$\\
= & $\val{\phi_1 \wedge \phi_2}^{\uparrow}(i)$\\
= & $\bigwedge_{f\in \mathsf{F}_{\mathbf{A}}(\mathbf{Fm})}[\val{\phi_1 \wedge \phi_2}(f)\to I(f, i)]$\\
= & $\bigwedge_{f\in \mathsf{F}_{\mathbf{A}}(\mathbf{Fm})}[f(\phi_1 \wedge \phi_2)\to I(f, i)]$.\\
\end{tabular}
\end{center}
Hence,  to show that $i(\phi_1 \wedge \phi_2)\leq \descr{\phi_1 \wedge \phi_2}(i)$, we need to show that for every $f\in \mathsf{F}_{\mathbf{A}}(\mathbf{Fm})$,
\[i(\phi_1 \wedge \phi_2)\leq f(\phi_1\wedge \phi_2)\to I(f, i).\]
Since by definition $I(f, i) = \bigvee_{\phi\in \mathbf{Fm}}[f(\phi)\otimes i(\phi)]\geq f(\phi_1 \wedge \phi_2)\otimes i(\phi_1 \wedge \phi_2)$ and $\to$ is order-preserving in the second coordinate, 
it is enough to show that for every $f\in \mathsf{F}_{\mathbf{A}}(\mathbf{Fm})$,
\begin{center}
\begin{tabular}{cl}
& $i(\phi_1 \wedge \phi_2)\leq f(\phi_1\wedge \phi_2)\to f(\phi_1 \wedge \phi_2)\otimes i(\phi_1 \wedge \phi_2)$.\\
\end{tabular}
\end{center}

By residuation, the above inequality is equivalent to
\begin{center}
\begin{tabular}{cl}
& $i(\phi_1 \wedge \phi_2)\otimes f(\phi_1\wedge \phi_2)\leq f(\phi_1 \wedge \phi_2)\otimes i(\phi_1 \wedge \phi_2),$\\
\end{tabular}
\end{center}
which is an instance of a tautology in residuated lattices.

To show that $\descr{\phi_1 \wedge \phi_2}(i)\leq i(\phi_1 \wedge \phi_2)$, it is enough to find some $f\in \mathsf{F}_{\mathbf{A}}(\mathbf{Fm})$ such that 
\[f(\phi_1 \wedge \phi_2)\to I(f, i)\leq i(\phi_1 \wedge \phi_2).\]
Let  $f_{\phi_1 \wedge \phi_2}:\mathbf{Fm}\to \mathbf{A}$ be defined by 
\[
f_{\phi_1 \wedge \phi_2}(\psi) = \left\{\begin{array}{ll}
1 & \text{if }  \phi_1 \wedge \phi_2\vdash \psi\\
0 & \text{otherwise. } \\
\end{array}\right.
\]
The inequality above then becomes
\[I(f_{\phi_1 \wedge \phi_2}, i)\leq i(\phi_1 \wedge \phi_2).\]
Indeed, 
by definition, $I(f_{\phi_1 \wedge \phi_2}, i) =  \bigvee_{\psi\in \mathbf{Fm}}[f_{\phi_1 \wedge \phi_2}(\psi)\otimes i(\psi)] =  \bigvee_{\phi_1 \wedge \phi_2\vdash \psi}[1\otimes i(\psi)]=  \bigvee_{\phi_1 \wedge \phi_2\vdash \psi}i(\psi)\leq i(\phi_1 \wedge \phi_2)$, the last inequality being due to the fact that  $i$ is order-reversing.

If  $\phi: = \phi_1\vee\phi_2$, then $\descr{\phi_1\vee\phi_2}(i) = (\descr{\phi_1} \wedge \descr{\phi_2})(i) = \descr{\phi_1}(i) \wedge \descr{\phi_2}(i) = i(\phi_1) \wedge i(\phi_2) =  i(\phi_1 \vee \phi_2)$. 
Let us  show that $\val{\phi_1 \vee \phi_2}(f) = f(\phi_1 \vee \phi_2)$. By definition,
\begin{center}
\begin{tabular}{cl}
& $\val{\phi_1 \vee \phi_2}(f)$\\
= & $\descr{\phi_1 \vee \phi_2}^{\downarrow}(f)$\\
= & $\bigwedge_{i\in \mathsf{I}_{\mathbf{A}}(\mathbf{Fm})}[\descr{\phi_1 \vee \phi_2}(i)\to I(f, i)]$\\
= & $\bigwedge_{i\in \mathsf{I}_{\mathbf{A}}(\mathbf{Fm})}[i(\phi_1 \vee \phi_2)\to I(f,i)]$.\\
\end{tabular}
\end{center}
Hence,  to show that $f(\phi_1 \vee \phi_2)\leq \val{\phi_1 \vee \phi_2}(f)$, we need to show that for every $i\in \mathsf{I}_{\mathbf{A}}(\mathbf{Fm})$,
\[f(\phi_1 \vee \phi_2)\leq i(\phi_1 \vee \phi_2)\to I(f, i).\]
Since by definition $I(f, i) = \bigvee_{\psi\in \mathbf{Fm}}[f(\psi)\otimes i(\psi)]\geq f(\phi_1 \vee \phi_2)\otimes i(\phi_1 \vee \phi_2)$ and $\to$ is order-preserving in the second coordinate, 
it is enough to show that for every $i\in \mathsf{I}_{\mathbf{A}}(\mathbf{Fm})$,
\begin{center}
\begin{tabular}{cl}
& $f(\phi_1 \vee \phi_2)\leq i(\phi_1 \vee \phi_2)\to f(\phi_1 \vee \phi_2)\otimes i(\phi_1 \vee \phi_2)$.\\
\end{tabular}
\end{center}
By residuation the inequality above is equivalent to
\begin{center}
\begin{tabular}{cl}
& $ f(\phi_1 \vee \phi_2)\otimes  i(\phi_1 \vee \phi_2)\leq f(\phi_1 \vee \phi_2)\otimes i(\phi_1 \vee \phi_2)$,\\
\end{tabular}
\end{center}
which is a tautology in residuated lattices.

To show that $\val{\phi_1 \vee \phi_2}(f)\leq f(\phi_1 \vee \phi_2)$, it is enough to find some $i\in \mathsf{I}_{\mathbf{A}}(\mathbf{Fm})$ such that 
\begin{equation}
\label{eqqqqq}
i(\phi_1 \vee \phi_2)\to I(f, i)\leq f(\phi_1 \vee \phi_2).\end{equation}
Let  $i_{\phi_1 \vee \phi_2}:\mathbf{Fm}\to \mathbf{A}$ be defined by the assignment 
\[
i_{\phi_1 \vee \phi_2}(\psi) = \left\{\begin{array}{ll}
1 & \text{if }  \psi\vdash \phi_1 \vee \phi_2 \\
0 & \text{if }\psi\not\vdash \phi_1 \vee \phi_2.
\end{array}\right.
\]
By definition and since $f$ is order-preserving and proper, $I(f, i_{\phi_1\vee\phi_2}) = \bigvee_{\psi\in \mathbf{Fm}}[f(\psi)\otimes i_{\phi_1 \vee \phi_2}(\psi)] = \bigvee_{\psi\vdash\phi_1 \vee \phi_2} [f(\psi)\otimes 1] = f(\phi_1\vee\phi_2)$. Hence, \eqref{eqqqqq} can be rewritten as follows:
\[1\to  f(\phi_1 \vee \phi_2)\leq f(\phi_1 \vee \phi_2),\]
which is true.

Assume $\phi: = \Diamond\psi$. We first show that $\descr{\Diamond\psi} (i) = i(\Diamond \psi)$. By definition,
\begin{center}
\begin{tabular}{rcl}
$\descr{\Diamond\psi} (i)$ &
 = & $R^{(0)}_\Diamond[\val{\psi}](i)$\\
&= & $\bigwedge_{f\in \mathsf{F}_{\mathbf{A}}(\mathbf{Fm})}[\val{\psi}(f)\to  R_{\Diamond}(i, f)]$\\
&= & $\bigwedge_{f\in \mathsf{F}_{\mathbf{A}}(\mathbf{Fm})}[f(\psi)\to R_{\Diamond}(i, f) ]$,\\
\end{tabular}
\end{center}
Hence,  to show that $i(\Diamond\psi)\leq \descr{\Diamond\psi}(i)$, we need to show that for every $f\in \mathsf{F}_{\mathbf{A}}(\mathbf{Fm})$,
\[i(\Diamond\psi)\leq f(\psi)\to R_{\Diamond}(i, f).\]
By definition we have  $R_\Diamond(i, f) = \bigvee_{\phi\in \mathbf{Fm}}(f(\phi) \otimes i(\Diamond\phi)) 
\geq f(\psi)\otimes i(\Diamond \psi)$, 
and since $\to$ is order-preserving in the second coordinate, 
it is enough to show that for every $f\in \mathsf{F}_{\mathbf{A}}(\mathbf{Fm})$,
\[i(\Diamond\psi)\leq f(\psi)\to (f(\psi)\otimes i(\Diamond \psi)).\]
By residuation the inequality above is equivalent to
\[ i(\Diamond \psi)\otimes f(\psi)\leq f(\psi)\otimes (\Diamond\psi),\]
which is a tautology in residuated lattices.

To show that $\descr{\Diamond\psi}(i)\leq i(\Diamond\psi)$, it is enough to find some $f\in \mathsf{F}_{\mathbf{A}}(\mathbf{Fm})$ such that
\begin{equation}
\label{eqqqqqqq}
f(\psi)\to R_{\Diamond}(i, f)\leq i(\Diamond\psi).\end{equation}
Let  $f_{\psi}:\mathbf{Fm}\to \mathbf{A}$ be defined by
\[
f_{\psi}(\phi) = \left\{\begin{array}{ll}
1 & \text{if }  \psi\vdash \phi\\
0 & \text{otherwise. } \\
\end{array}\right.
\]
By definition and Lemma \ref{eq:premagicnew2},
\begin{center}
\begin{tabular}{rcl}
$R_\Diamond(i, f_{\psi})$ &  = &$\bigvee_{\phi\in \mathbf{Fm}}(f_{\psi}^{-\Diamond}(\phi)\otimes  i(\phi)) $\\
&=&$ \bigvee_{\phi\in \mathbf{Fm}}(f_{\psi}(\phi)\otimes  i(\Diamond \phi))$\\
&=&$ \bigvee_{\psi\vdash \phi} i(\Diamond \phi)$\\
&$\leq$&$  i(\Diamond \psi)$,\\
\end{tabular}
\end{center}
the last inequality being due to the fact that $i$ and $\Diamond$ are order-reversing and order-preserving respectively. Since $\to$ is order-preserving in the second coordinate, to show that \eqref{eqqqqqqq} holds, it is enough to show that 
\[f_{\psi}(\psi)\to  i(\Diamond \psi)\leq i(\Diamond\psi).\]
This immediately follows from the fact that, by construction,  $f_{\psi}(\psi) = 1$.

Let us now show that $\val{\Diamond\psi}(f) = f(\Diamond\psi)$. By definition,
\begin{center}
\begin{tabular}{cl}
&$\val{\Diamond\psi} (f)$\\
  = & $\descr{\Diamond\psi}^{\downarrow}(f)$\\ 
 = & $\bigwedge_{i\in \mathsf{I}_{\mathbf{A}}(\mathbf{Fm})}[\descr{\Diamond\psi}(i)\to I(f, i)]$\\
 = & $\bigwedge_{i\in \mathsf{I}_{\mathbf{A}}(\mathbf{Fm})}[i(\Diamond\psi)\to I(f, i)].$\\
\end{tabular}
\end{center}
Hence,  to show that $f(\Diamond\psi)\leq \val{\Diamond\psi}(f)$, we need to show that for every $i\in \mathsf{I}_{\mathbf{A}}(\mathbf{Fm})$,
\[f(\Diamond\psi)\leq i(\Diamond\psi)\to I(f, i).\]
Since by definition $I(f, i) = \bigvee_{\phi\in \mathbf{Fm}}[f(\phi)\otimes i(\phi)]\geq f(\Diamond\psi)\otimes i(\Diamond\psi)$ and $\to$ is order-preserving in the second coordinate, 
it is enough to show that for every $i\in \mathsf{I}_{\mathbf{A}}(\mathbf{Fm})$,
\[f(\Diamond\psi)\leq i(\Diamond\psi)\to (f(\Diamond\psi)\otimes i(\Diamond\psi)),\]
which holds in all residuated lattices.

To show that $\val{\Diamond\psi}(f)\leq f(\Diamond\psi)$, it is enough to find  some $i\in \mathsf{I}_{\mathbf{A}}(\mathbf{Fm})$ such that 
\begin{equation}
\label{eqqqqqqqq}u(\Diamond\psi)\to I(f, i)\leq f(\Diamond\psi).\end{equation}
Let  $i_{\Diamond\psi}:\mathbf{Fm}\to \mathbf{A}$ be defined by the assignment 
	\[
i_{\Diamond\psi}(\phi) = \left\{\begin{array}{ll}

1 & \text{if }   \phi\vdash \Diamond\psi \\
0 & \text{if }\phi\not\vdash \Diamond\psi.
\end{array}\right.
\]
By definition and since $f$ is order-preserving and proper, $I(f, i_{\Diamond\psi}) = \bigvee_{\phi\in \mathbf{Fm}}[f(\phi)\otimes i_{\Diamond\psi}(\phi)] = \bigvee_{\phi\vdash\Diamond\psi} [f(\phi)\otimes 1 ] = f(\Diamond\psi)$. Hence, \eqref{eqqqqqqqq} can be rewritten as follows:
\[u_{\Diamond\psi}(\Diamond\psi)\to f(\Diamond\psi)\leq f(\Diamond\psi),\]
which is true since $i_{\Diamond\psi}(\Diamond\psi)=1$.

Assume $\phi:=\Box \psi$. We first show that $\val{\Box \psi}(f)=f(\Box \psi)$. By definition:
\begin{center}
\begin{tabular}{rcl}
$\val{\Box\psi} (f)$ & = & $R^{(0)}_\Box[\descr{\psi}](f)$\\
& = & $\bigwedge_{i \in \mathsf{I}_{\mathbf{A}}(\mathbf{Fm})} [\descr{\psi}(i) \to R_{\Box}(f,i)]$\\
& = & $\bigwedge_{i \in \mathsf{I}_{\mathbf{A}}(\mathbf{Fm})} [i(\psi) \to R_{\Box}(f,i)]$
\end{tabular}
\end{center}
Hence to show that $f(\Box \psi) \leq \val{\Box \psi}(f)$ we must show that for every $i \in \mathsf{I}_{\mathbf{A}}(\mathbf{Fm})$ we have
\[f(\Box \psi) \leq i(\psi) \to R_{\Box}(f,i).\] 

We have 
$R_{\Box}(f,i)=\bigvee_{\phi \in \mathbf{Fm}} (f(\Box \phi) \otimes i(\phi))$ so 
$R_{\Box}(f,i)\geq f(\Box \psi) \otimes i(\psi)$. 
Since $\to$ is order-preserving in the second coordinate, it will be enough to show that 
\[ f(\Box \psi) \leq i(\psi) \to (f(\Box \psi) \otimes i(\psi)).\]
Using residuation we can see that this is an instance of a tautology in residuated lattices. 

To show $\val{\Box \psi}(f)\leq f(\Box \psi)$ we must find $i \in \mathsf{I}_{\mathbf{A}}(\mathbf{Fm})$ such that 

\begin{equation}\label{eq:Box1}
 i(\psi) \to R_{\Box}(f,i) \leq f(\Box \psi).
\end{equation}
Let  $i_{\psi}:\mathbf{Fm}\to \mathbf{A}$ be defined by the assignment 
	\[
i_{\psi}(\chi) = \left\{\begin{array}{ll}
1 & \text{if }   \chi\vdash  \psi \\
0 & \text{if }\chi\not\vdash \psi.
\end{array}\right.
\]
By definition, and since $f$ is order-preserving, $R_{\Box}(f, i_{\psi}) = \bigvee_{\chi\in\mathbf{Fm}}(f(\Box\chi)\otimes i(\chi)) = \bigvee_{\chi\vdash \psi}(f(\Box\chi)\otimes 1) = f(\Box\psi)$. Hence, \eqref{eq:Box1} can be rewritten as follows:
\[i_{\psi}(\psi)\to f(\Box\psi) \leq f(\Box \psi),\]
which is a tautology since $i(\psi)=1$. 

Next, we want to show that $\descr{\Box \psi}(i)=i(\Box\psi)$. By definition,
\begin{center}
\begin{tabular}{rcl}
$\descr{\Box\psi} (i)$ & = & $\val{\Box\psi}^{\uparrow}(i)$\\
& = & $\bigwedge_{f \in \mathsf{F}_{\mathbf{A}}(\mathbf{Fm})} [\val{\Box \psi}(f) \to I(f,i)]$\\
& = & $\bigwedge_{i \in \mathsf{F}_{\mathbf{A}}(\mathbf{Fm})} [f(\Box \psi) \to I(f,i)].$
\end{tabular}
\end{center}
To show that $\descr{\Box \psi}(i) \leq i(\Box\psi)$ we just need to find $f \in \mathsf{F}_{\mathbf{A}}(\mathbf{Fm})$ such 
that $f(\Box \psi) \to I(f,i) \leq i(\Box \psi)$. Define $f_{\Box \psi} \colon \mathbf{Fm} \to \mathbf{A}$
by
\[ f_{\Box \psi}(\chi) = \begin{cases} 1 & \text{ if } \Box \psi \vdash \chi \\ 0 & \text{otherwise.}\end{cases} \]
Clearly $f_{\Box\psi}(\Box \psi)=1$ and so 
$f_{\Box\psi}(\Box \psi) \to I(f_{\Box\psi},i) = I(f_{\Box\psi},i)$. Now,
\begin{center}
\begin{tabular}{rcl}
$I(f_{\Box\psi},i)$ & = & $\bigvee_{\phi \in \mathbf{Fm}} (f_{\Box\psi}(\phi) \otimes i(\phi))$ \\ 
& = & $\bigvee_{\Box \psi \vdash \chi} (f_{\Box\psi}(\chi) \otimes i(\chi))$ \\
& = & $\bigvee_{\Box \psi \vdash \chi} (1\otimes i(\chi)) $ \\
& = & $\bigvee_{\Box \psi \vdash \chi} i(\chi) $ \\
& $\leq$ & $i(\Box \psi)$.
\end{tabular}
\end{center}
The last inequality follows from the fact that $i$ is order-reversing. Hence,
\[f_{\Box\psi}(\Box \psi) \to I(f_{\Box\psi},i) = I(f_{\Box\psi},i)\leq i(\Box \psi).\]
 
To show that $ i(\Box\psi)\leq \descr{\Box \psi}(i)$, we must show that 
for all $f \in \mathsf{F}_{\mathbf{A}}(\mathbf{Fm})$ we have $i(\Box\psi)\leq f(\Box \psi) \to I(f,i)$. By definition 
we have $I(f,i)=\bigvee_{\phi \in \mathbf{Fm}} (f(\phi) \otimes i(\phi)) \geq f(\Box \psi)\otimes i(\Box \psi)$. 
Therefore, the desired inequality will follow if we can show 
\[
i(\Box\psi) \leq f(\Box \psi) \to (f(\Box \psi)\otimes i(\Box \psi)),
\]
which is an instance of a tautology in residuated lattices. 
\end{proof}


\begin{theorem} \label{thm:completeness}
The basic normal $\mathcal{L}$-logic $\mathbf{L}$ is sound and complete w.r.t.~the class of polarity-based $\mathbf{A}$-frames. 
\end{theorem}
\begin{proof}
Consider an $\mathcal{L}$-sequent $\phi \vdash \psi$ that is not derivable in $\mathbf{L}$.
In order to show that $\mathbb{M} \models \phi \vdash \psi$, we
need to show that $\val{\phi}(f) \nleq \val{\psi}(f)$ for some $f \in \mathsf{F}_{\mathbf{A}}(\mathbf{Fm})$. 
 Consider the proper filter $f_{\phi}$ and proper ideal $i_{\psi}$ given by 
\[
f_{\phi}(\chi) = \left\{\begin{array}{ll}
1 &\text{if } \phi \vdash \chi\\
0 &\text{if } \phi \not \vdash \chi
\end{array} \right.
\]
and
\[
i_{\psi}(\chi) = \left\{\begin{array}{ll}
1 &\text{if } \chi \vdash \psi\\
0 &\text{if } \chi \not \vdash \psi.
\end{array} \right.
\]
Then $\bigvee_{\chi\in \mathbf{Fm}}(f_{\phi}(\chi)\otimes i_{\psi}(\chi)) = 0$, for else there would have to be a formula $\chi_0 \in \mathbf{Fm}$ such that $f_{\phi}(\chi_0) = 1$ and $i_{\psi}(\chi_0) = 1$, which would mean that $\phi \vdash \chi_0$ and $\chi_0 \vdash \psi$ and hence that $\phi \vdash \psi$, in contradiction with the assumption that $\phi \vdash \psi$ is not derivable. By the Truth Lemma, $\val{\phi}(f_\phi) = f_{\phi}(\phi) = 1$ and $\descr{\psi}(i_\psi) = i_{\psi}(\psi) = 1$, and so 
\begin{eqnarray*}
\val{\psi}(f_\phi) & = & (\descr{\psi})^\downarrow(f_\phi)\\
& = & \bigwedge_{i\in \mathsf{I}_{\mathbf{A}}(\mathbf{Fm})}(\descr{\psi}(i) \to I(f_\phi, i))\\
& \leq & (\descr{\psi}(i_\psi) \to I(f_\phi, i_\psi))\\
& = & i_\psi(\psi) \to I(f_\phi, i_\psi))\\
& = & 1 \to I(f_\phi, i_\psi))\\
& = & I(f_\phi, i_\psi))\\
& = & \bigvee_{\chi\in \mathbf{Fm}}(f_{\phi}(\chi)\otimes i_{\psi}(\chi))\\
& = & 0.
\end{eqnarray*}
So we conclude that  $\mathbb{M} \not \models \phi \vdash \psi$, as desired.
\end{proof}

\end{document}

%% file: packages.tex
\usepackage{lscape}
\usepackage{pdflscape}

\usepackage{stmaryrd}

\usepackage{cmll}
\usepackage[table]{xcolor}
\usepackage{comment}

\usepackage{mathdots}

\usepackage{amscd}
\usepackage{amssymb}
\usepackage{amsmath}
\usepackage{mathptmx}
\usepackage{mathrsfs}
\usepackage{color}
\usepackage{xspace}
\usepackage{bussproofs}
\EnableBpAbbreviations

\usepackage{tikz}

\usepackage{txfonts}

\usepackage{centernot}

\usepackage{mathtools}

\usepackage{relsize}

\usepackage{multirow}

\usepackage{multicol}

\usepackage{array}
\newcolumntype{C}[1]{>{\centering\arraybackslash}p{#1}}
\newcolumntype{L}[1]{>{\arraybackslash}p{#1}}

\usepackage{hyperref}